%% file: main.tex
\documentclass[12pt]{amsart}

\usepackage[margin=1in]{geometry}
\usepackage[T1]{fontenc}
\usepackage{lmodern}
\usepackage{amsmath}
\usepackage{amssymb}
\usepackage{mathtools}
\usepackage{graphicx}
\usepackage{url}
\usepackage[numbers,square,sort&compress]{natbib}
\usepackage[ruled,vlined]{algorithm2e}
\usepackage[shortlabels]{enumitem}
\usepackage{xcolor}
\usepackage{thmtools}
\usepackage{thm-restate}
\usepackage[colorlinks=true,linkcolor=blue,citecolor=blue,urlcolor=blue]{hyperref}
\usepackage[capitalise]{cleveref}

\title[Matrix concentration inequalities for time-inhomogeneous Markov chains]{Matrix concentration inequalities for time-inhomogeneous Markov chains}
\author{Luca Zanetti}
\address{Department of Mathematical Sciences, University of Bath, UK}
\email{lz2040@bath.ac.uk}
\date{\today}

\theoremstyle{plain}
\newtheorem{theorem}{Theorem}
\newtheorem{lemma}[theorem]{Lemma}
\newtheorem{proposition}[theorem]{Proposition}

\theoremstyle{definition}
\newtheorem{definition}[theorem]{Definition}

\theoremstyle{remark}
\newtheorem{remark}[theorem]{Remark}

\newcommand{\E}{\mathbb{E}}
\newcommand{\Pbb}{\mathbb{P}}
\newcommand{\Pbbmu}{\mathbb{P}_{\mu_0}}

\newcommand{\R}{\mathbb{R}}
\newcommand{\C}{\mathbb{C}}
\newcommand{\Hm}{\mathsf{H}_m}
\newcommand{\Id}{I_m}
\newcommand{\op}{\mathrm{op}}
\newcommand{\tr}{\mathrm{tr}}
\newcommand{\Lip}{\mathrm{Lip}}
\newcommand{\diam}{\mathrm{diam}}
\newcommand{\rank}{\mathrm{rank}}
\newcommand{\lmax}{\lambda_{\max}}

\newcommand{\normtwo}[2]{\left\|#1\right\|_{2,#2}}
\newcommand{\norminf}[1]{\left\|#1\right\|_{\infty}}

\newcommand{\supp}{\mathrm{Supp}}
\newcommand{\osc}{osc}

\renewcommand{\Pr}[1]{\mathbb{P}\left[\,#1\,\right]}

\newtheorem{assumption}{Assumption}

\begin{document}

\begin{abstract}
We establish Chernoff-type bounds for the largest eigenvalue of sums of Hermitian random matrices generated by a time-inhomogeneous Markov chain. Our primary regime assumes a compact state space and contractivity of each Markov kernel in Wasserstein distance, i.e., positive Ollivier-Ricci curvature.  This assumption is convenient to verify in inhomogeneous settings and is satisfied by several chains of practical interest, such as stochastic gradient descent on strongly convex smooth objectives. We also develop analogous bounds for noncompact state spaces under a notion of spectral gap for inhomogeneous chains introduced by Saloff-Coste and Zuniga. Finally, we illustrate the utility of our results through an analysis of the Elo rating system, a popular method for ranking players in sports analytics, under a dynamic version of the Bradley-Terry-Luce model.
\end{abstract}



\maketitle
\vspace{-1em}

\tableofcontents
\newpage

\input{intro}

\input{preliminaries}

\section{Main results}
\label{sec:main_results}
\input{curv}

\input{spectral}

\input{proofs}

\input{elo}

\begingroup
\makeatletter
\let\@tocwrite\@gobbletwo
\section*{Acknowledgements}
\makeatother
\endgroup
The author would like to thank Sam Olesker-Taylor for many insightful conversations about positively curved Markov chains and the Elo rating system.
\vspace{0.5cm}

\providecommand{\doi}[1]{%
  \href{https://doi.org/#1}{https://doi.org/#1}%
}
\bibliographystyle{plainnat}
\bibliography{ref}

\appendix

\input{useful}

\input{kappa_squared}

\end{document}

%% file: intro.tex
\section{Introduction}
Concentration inequalities are fundamental tools in the design and analysis of learning and randomised algorithms.
Although originally developed for sums of i.i.d. real-valued random variables,  substantial progress over the last three decades has extended this theory to (i) dependent, in particular Markov-dependent, random variables \citep{gillman,L:chernoff-markov-disc,JoulinOllivier,P:conc-psgap}; (ii) high-dimensional observables, notably vector- and matrix-valued random variables~\citep{AW02,userfriendly}.
More recently, \citet{GLSS} combined these lines of work to obtain a Chernoff bound for matrix-valued observables of a finite Markov chain. This has been further generalised to arbitrary state spaces and time-dependent observables by \citet{NSW}.

Most of the progress in the Markovian setting, however, has focused on time-homogeneous chains: homogeneity allows one to exploit approaches based on spectral methods~\citep{gillman,L:chernoff-markov-disc,LeonPerron,Healy,GLSS,NSW} or the Poisson equation~\citep{MR4982520,durmus2023rosenthal,peng2025matrix}, which do not readily extend to the inhomogeneous setting.

Motivated by this gap, we study concentration inequalities for matrix-valued observables on time-inhomogeneous Markov chains.
Our main focus will be on chains contractive in Wasserstein distance, i.e., positively curved according to \citet{OllivierJFA}. While positive curvature is a strong notion of convergence and is not ubiquitously satisfied, several important Markov chains arising in algorithm design do exhibit positive curvature, such as Glauber dynamics for the high-temperature Ising model~\citep{OllivierJFA} and stochastic gradient descent on strongly convex and smooth functions~\citep{SGDMarkov} (see \citep{EberleMajka,JoulinOllivier} for further examples). Crucially, positive curvature is a \emph{local-in-time} property, in the sense that it can be verified for each Markov kernel $P_t$ individually, rather than requiring control of the entire evolution of the chain. This makes it well suited to the time-inhomogeneous setting, where it is often difficult to establish properties about the global dynamics of the chain.

To state our first main result in an informal fashion, consider a (possibly inhomogeneous) Markov chain $(X_t)_{t \ge 0}$ on a Polish metric space $(\Omega,d)$ with bounded diameter $D$, whose Markov kernels $(P_t)_{t \ge 1}$ contract in Wasserstein distance by a factor $(1-\kappa)$. Let $F_t$ be $m \times m$ Hermitian-valued observables with uniform Lipschitz constant
$\Lip(F_t)\le L$ (where the Lipschitz constant is defined with respect to the operator norm).  We establish the following concentration inequality on the largest eigenvalue of the empirical averages of these observables:
\begin{equation}
\label{eq:main}
\Pbb\Big(\lmax\Big(\sum_{t=1}^n (F_t(X_t) - \E{F_t(X_t)}) \Big)\ge n\varepsilon\Big)\le
m^{\,2-\frac{\pi}{4}}\,
\exp\left(-\frac{Cn \kappa\varepsilon^2}{L^2D^{2}}\right)
\qquad \text{for all } \varepsilon\ge0,
\end{equation}
where $C>0$ is some universal constant. 
Essentially, we obtain sub-Gaussian concentration with variance proxy \(O(L^2D^2 \kappa^{-1})\).

To prove \eqref{eq:main}, we control the matrix moment generating function via the multivariate
trace inequality of \citet{GLSS} (a many-matrix extension of Golden--Thompson), reducing the problem to bounding the trace of a product of Markov-dependent matrices. We then generalise the ``direct method'' of \citet{L:chernoff-markov-disc} to the matrix and inhomogeneous setting: the main idea of this technique, which Lezaud attributes to Bakry and Ledoux, is to iteratively centre the functions on which the operators $P_1, \dots, P_{n-1}$ act. Our simple but key observation is that,  for a centred matrix-valued $\widetilde F$, we can bound $\sup_x \|\widetilde F(x)\|_{\op}$ by the diameter and the Lipschitz constant of $\widetilde F$. We then exploit the fact that positive curvature implies each kernel $P_t$ is contractive in the Lipschitz seminorm, and this contraction extends to matrix-valued functions as well.

One of the limitations of \cref{eq:main} is that it depends on the diameter of the space. In \cref{thm:curv_diam} we improve the dependency on the diameter to logarithmic, obtaining a variance proxy of order $\Delta_\op^2 \kappa^{-1} (1 + \log \frac{LD}{\Delta_\op})$, where $\Delta_\op \le LD$ is a bound (in operator norm) on the oscillation of $F_t$.

We remark that a factor depending on the diameter and the Lipschitz constant of the functions is needed for any approach relying only on curvature: there are examples where $\kappa = \Omega(1)$, the functions are nicely bounded, but the asymptotic variance is unbounded. Indeed, a logarithmic factor in $LD$ is needed in the worst case, even assuming $\Delta_\op \le 1$ (see \cref{app:tight} for an example). If no assumption on $\Delta_\op$ is imposed, a variance proxy of $L^2D^2 \kappa^{-1}$ is optimal.

To generalise \cref{eq:main} to potentially unbounded state spaces, we consider inhomogeneous kernels contractive on centred functions as operators $L^2(\mu_t) \to L^2(\mu_{t-1})$, where $\mu_t$ denotes the $t$-step distribution of the chain started at $\mu_0$. This allows us to replace the diameter factor by a uniform bound on the norm of the observables. This notion of contractivity, which essentially generalises the spectral gap to inhomogeneous chains, has been investigated by \citet{SCZ} in relation to the \emph{merging time}, i.e., the time a chain takes to forget its initial state.  While this allows us to remove the dependency on the diameter, this notion of contraction can be difficult to verify in practice, as it requires detailed control of the chain’s evolution.

We conclude by discussing an application of our results to the analysis of the Elo rating system~\citep{A:elo-rats, EloOurs}, a widely used method for estimating
the relative skill of players in sports analytics, especially chess and tennis. We study Elo under the Bradley-Terry-Luce model, a standard statistical model for match outcomes, extended to a dynamical setting in which players' true skills can change over time. In this setting, Elo induces a time-inhomogeneous Markov chain: our results imply that, if the true skills change slowly enough, Elo ratings are able to track the evolving skills over time.

In general, our results are useful to analyse learning algorithms where data distributions evolve over time; for instance, they apply to stochastic gradient descent on time-varying strongly convex smooth objectives and to PCA with Markovian data~\citep{StreamingPCAMarkov}, where samples are generated by a time-inhomogeneous chain.

\vspace{0.5cm}
\textbf{Additional related work}
Compared to the homogeneous setting, concentration inequalities for time-inhomogeneous chains are much less explored.
\citet{pillai2014finitesamplepropertiesadaptive} have explicitly adapted the Markov chain Chernoff bounds of \citet{JoulinOllivier} for positively curved chains to the inhomogeneous (scalar)  setting. Their techniques, however, obtain a variance proxy of order $L^2 \sigma_\infty^2 \kappa^{-2}$, where $\sigma_\infty \coloneq \sup_{t,x} \diam \supp P_t(x,\cdot)$ is the granularity of the chain. 
This can be much worse than our bounds whenever mixing is much slower than the diameter of the chain.
For completeness,  we also generalise their techniques to matrix observables in  \cref{app:ollivier}.

The martingale method of \citet{MR2478678} and \citet{MR3837271} also applies to empirical means of scalar observables in inhomogeneous chains. Their techniques, however, typically require stronger mixing properties than positive curvature.

%% file: preliminaries.tex
\section{Preliminaries}
\label{sec:preliminaries}
\subsection{Notation}
Let $\Omega$ be a general state space with $\sigma$-algebra $\mathcal{B}$.
A time-inhomogeneous Markov chain \((X_t)_{t\ge 0}\) is specified by an initial distribution \(\mu_0\) and a sequence of Markov kernels \(P_t\) (\(t\ge 1\)), so that the transition from time \(t-1\) to \(t\) satisfies the following:
\[
\Pbb(X_t\in A\mid X_{t-1}=x)=P_t(x,A),\qquad  \forall \, t\ge 1, \, A \in \mathcal{B}.
\]
For integers \(1\le s\le t\), we denote the inhomogeneous product kernel by
\(
P_{s:t}:=P_sP_{s+1}\cdots P_t.
\)
The probability distribution at time $t$ is given by $\mu_t = \mu_0 P_{1:t}$, where $\mu_t$ is defined recursively as
\[
\mu_t(dy) = \int_\Omega P_t(x,dy) \mu_{t-1}(dx).
\]
Unless explicitly stated otherwise, all expectations and probabilities are taken under the path law started from \(X_0\sim\mu_0\);
we write \(\E_{\mu_0}[\cdot]\) and \(\Pbbmu(\cdot)\) when it is helpful to emphasise this.

For any measurable function $f \colon \Omega \to \R$, $P_t$ acts on $f$ as
\[
P_t f(x) = \int_\Omega f(y) P_t(x,dy) \qquad \forall \, x \in \Omega.
\]
We denote by $\Hm$ the set of $m \times m$ Hermitian matrices.
We use upper case letters for matrix-valued functions $F \colon \Omega \to \mathbb{C}^{m \times m}$.
With some abuse of notation, we let $P_t$ act on $F$ as
\[
P_t F(x) = \int_\Omega F(y) P_t(x,dy) \qquad \forall \, x \in \Omega.
\]

Given a vector $v \in \mathbb{C}^m$, we denote its Euclidean norm by
$\|v\| \coloneq \sqrt{ \sum_{i=1}^m |v_i|^2}$.
 For a matrix $A \in \mathbb{C}^{m \times m}$, we denote its operator norm as
$\|A\|_{\op} \coloneqq \sup_{x \in \mathbb{C}^m \setminus \{\mathbf{0}\} } \frac{\|Ax\|}{\|x\|}$.
The Frobenius norm of $A$ is denoted as $\|A\|_{\mathrm F} \coloneqq \sqrt{\tr A^\ast A}$,
 where $A^\ast$ is the complex adjoint of $A$.
If $A \in \Hm$, we denote its largest eigenvalue as $\lambda_{\max}(A)$.
For $F \colon \Omega \to \mathbb{C}^{m \times m}$, we define $\|F\|_\infty \coloneqq \sup_{x \in \Omega} \|F(x)\|_\op$.

\subsection{Proof outline and preliminary results}
Let $(F_t)_{t\ge 1}$ be a sequence of functions $F_t \colon \Omega \to \Hm$.
We are interested in high probability bounds for objects of the following form:
\[
\lambda_{\max}\left(\sum_{t=1}^n (F_t(X_t) - \E F_t(X_t))  \right).
\]

As standard in matrix concentration inequalities~\citep{userfriendly}, our proof relies on bounding the matrix moment generating function:
for any random \(Z\in\Hm\), any \(s>0\), and any \(a\in\R\),
\begin{equation}\label{eq:laplace_maxeig}
\Pr{\lambda_{\max}(Z)\ge a} \le e^{-sa}\,\E\tr e^{sZ}.
\end{equation}

In our case, $Z$ is a sum of Markov-dependent random matrices. To bound this moment generating function,
our starting point is the recent generalisation of the Golden-Thompson inequality to several matrices, which is due to \citet{GLSS} building on work by \citet{MultiTraceIneq}.
\begin{theorem}[\citet{GLSS}]\label{thm:GLSS}
Let \(H_{1},\dots,H_n\in\Hm\). Then, there exists a probability measure $\nu$ on  $[-\pi/2,\pi/2]$ such that
\begin{equation}\label{eq:GLSS}
\tr\exp\Big(\frac{\pi}{4} \, \sum_{t=1}^n H_t\Big)
\;\le\;
 m^{1-\pi/4}\,
 \int_{-\pi/2}^{\pi/2}
\tr\Bigg(\prod_{t=1}^n e^{\frac{1}{2}e^{i\phi}H_t} \,
\prod_{t=n}^1 e^{\frac{1}{2}e^{-i\phi}H_t}\Bigg)\,\nu(\phi).
\end{equation}
\end{theorem}

Inequality \eqref{eq:GLSS} allows us to upper bound the matrix mgf as a trace of products of matrix exponentials, which is somewhat more similar to the traditional mgf for a sum of independent scalar random variables. We will obtain bounds on this quantity in two alternative regimes:
\begin{enumerate}[(A)]
\item Lipschitz observables and Wasserstein-contractive Markov operators on compact state spaces (\cref{sec:curv});
\item observables bounded in operator norm and Markov operators contractive from $L^2(\mu_t) \to L^2(\mu_{t-1})$ on general spaces (\cref{sec:spectral}).
\end{enumerate}

We employ the same proof strategy in both regimes, which generalises the ``direct method'' of \citet{L:chernoff-markov-disc} to the inhomogeneous and noncommutative setting. Here, ``direct'' is in contrast to proof strategies requiring results in the perturbation theory of linear operators, such as the original Markov chain Chernoff bound by \citet{gillman}.

Fix \(n\ge 1\), \(\phi\in[-\pi/2,\pi/2]\), and \(s\ge 0\).
Define matrices
\[
W_t(x):=\exp\!\Big(\frac12 e^{i\phi}\,s\,\widetilde F_t(x)\Big),
\qquad
W_t(x)^\ast=\exp\!\Big(\frac12 e^{-i\phi}\,s\,\widetilde F_t(x)\Big),
\]
where $\widetilde F_t \coloneq F_t - \mu_t(F_t)$,
and their products
\[
M_j:=W_1(X_1)\,W_{2}(X_{2})\cdots W_{j}(X_{j}),\qquad j=1,\dots,n,
\]
with \(M_0=\Id\).
Let
\begin{equation}\label{eq:an_def}
a_j(\phi):=\E\tr(M_jM_j^\ast).
\end{equation}

Our goal is to bound $\sup_{\phi} a_n(\phi)$: we do this by first establishing a recursive inequality,
which is the time-inhomogeneous matrix version of the recursive identity by \cite{L:chernoff-markov-disc}.
The key idea is, for any $1 \le i \le n$, to centre
the term on which each successive operator $P_i$ acts. This will allow us to exploit the contractive properties
(in Lipschitz or $L^2$ norm) of $P_i$.

\begin{restatable}{proposition}{Renewal}\label[proposition]{prop:renewal}
Let \(\phi\in[-\pi/2,\pi/2]\), \(\gamma:=\cos\phi\in[0,1]\), and $n \ge 1$.
Define, for $1 \le i \le n$, $B_{n,i} \in \Hm$ and  $\Theta_{n,i}, H_{n,i} \colon \Omega \to \Hm$
as follows:
\begin{align*}
\Theta_{n,1}(x) &:= \exp\!\big(s\gamma\,\widetilde F_n(x)\big),
\qquad
B_{n,1}:=\E\big[\Theta_{n,1}(X_n)\big],
\qquad
H_{n,1}(x):=\Theta_{n,1}(x)-B_{n,1}, \\
\Theta_{n,i}(x) &:= W_{n-i+1}(x)\,\big(P_{n-i+2}H_{n,i-1}\big)(x)\,W_{n-i+1}(x)^\ast,
\qquad
B_{n,i}:=\E\big[\Theta_{n,i}(X_{n-i+1})\big], \\
H_{n,i}(x) &:= \Theta_{n,i}(x)-B_{n,i},\qquad i=2,\dots,n.
\end{align*}
Let $b_{n,i}(\phi)=\|B_{n,i}\|_{\op} \, (1 \le i \le n)$. Then,
\begin{equation}\label{eq:renewal_ineq}
a_n(\phi)\ \le\ \sum_{i=1}^n b_{n,i}(\phi)\,a_{n-i}(\phi).\end{equation}
\end{restatable}

The next lemma shows that a bound on the sum of the coefficients $(b_{n,i})_{i=1}^n$ translates into a bound on the matrix moment generating function.

\begin{restatable}{lemma}{TraceBound}\label[lemma]{lem:close}
Assume that for some \(C\ge 0\), for all \(n\ge 1\) and all \(\phi \in [-\pi/2,\pi/2]\),
\begin{equation}\label{eq:beta_sum}
\sum_{i=1}^n b_{n,i}(\phi)\ \le\ 1 + C s^2.
\end{equation}
Then for all \(n\ge 0\), \(a_n(\phi)\le m\,(1+Cs^2)^n\le m\,\exp(Cns^2)\).
Therefore,
\begin{equation}\label{eq:mgf_generic}
\E_{\mu_0}\tr\exp\!\Bigl(\frac{\pi}{4}\sum_{j=1}^n s\,\widetilde F_j(X_j)\Bigr)\le m^{2-\pi/4}\exp\!\Big(Cn s^2\Big).
\end{equation}
\end{restatable}

The previous lemma, together with the standard matrix Chernoff bound argument of \cref{eq:laplace_maxeig} and \cref{thm:GLSS}, implies that we only need to control the coefficients $(b_{n,i})_{i=1}^n$ to obtain a concentration inequality. We outline how to do that in two different regimes: (A) for positively curved chains (\cref{sec:curv}); (B) for chain contractive in $L^2(\mu_t) \to L^2(\mu_{t-1})$ (\cref{sec:spectral}).

%% file: curv.tex
\subsection{Regime (A): Positive curvature and Lipschitz observables}
\label{sec:curv}
We consider time-inhomogeneous chains with positive curvature, bounded diameter, and Lipschitz observables.
Throughout this section, we assume our state space $\Omega$ is a Polish space with a metric $d$.
We recall the definition of Wasserstein-1 distance.

 \begin{definition}[Wasserstein-1 distance]
Let $(\Omega, d)$ be a Polish metric space and let $\mu,\nu$ be two probability measures on $\Omega$ with finite first moment.
The \emph{Wasserstein-1 distance} between $\mu$ and $\nu$ is
\[
    W_1(\mu,\nu)
    \;:=\;
    \inf_{\pi \in \Pi(\mu,\nu)}
    \int_{\Omega\times \Omega} d(x,y)\,\mathrm{d}\pi(x,y),
\]
where $\Pi(\mu,\nu)$ is the set of all couplings of $\mu$ and $\nu$,
i.e., all probability measures $\pi$ on $\Omega\times\Omega$
with marginals $\mu$ and $\nu$.
\end{definition}

By Kantorovich–Rubinstein  duality, we have:
\[
    W_1(\mu,\nu)
    \;=\;
    \sup_{\Lip(f) \le 1}
    \biggl(
        \int_{\Omega} f(x)\,\mathrm{d}\mu(x)
        -
        \int_{\Omega} f(x)\,\mathrm{d}\nu(x)
    \biggr),
\]
where the Lipschitz seminorm of $f:\Omega\to\mathbb{R}$ is
defined by
\[
    \Lip(f)
    \;:=\;
    \sup_{x \neq y}
    \frac{|f(x)-f(y)|}{d(x,y)}.
\]

We also generalise the definition of Lipschitzness to matrix-valued functions as follows.
\begin{definition}[Operator Lipschitz seminorm]\label{def:Lip_op}
For \(F:\Omega\to\mathbb{C}^{m \times m}\),
\[
\Lip(F)^{\op}:=\sup_{x\neq y}\frac{\|F(x)-F(y)\|_{\op}}{d(x,y)}.
\]
\end{definition}

We follow \citet{OllivierJFA} and define the curvature of a Markov chain as the worst case contraction in Wasserstein distance after one step.  Since we are dealing with time-inhomogeneous Markov chains, we specialise our notion for each time step.

\begin{definition}[Ollivier curvature]
\label[definition]{def:curvature}
For any integer $t \ge 1$, we define the curvature at time $t$ as
\[
    \kappa_t
    \;:=\;
    \inf_{x \ne y} 1 - \frac{W_1(P_t(x, \cdot), P_t(y, \cdot))}{d(x,y)}.
\]
\end{definition}

By Kantorovich–Rubinstein duality, a positive curvature implies that the averaging operator is contractive in the Lipschitz seminorm:
\[
\Lip(P_tf) \le (1-\kappa_t) \Lip(f).
\]

Since the curvature of the chain can vary greatly between steps, we only assume each step of the Markov chain of interest has only non-negative curvature and we assume strictly positive curvature ``on average''.

\begin{assumption}[Positive curvature]\label[assumption]{ass:curv}
For all \(t\ge 1\), we have \(\kappa_t \in [0,1]\). Furthermore, there exists \(\kappa \in (0,1]\) such that, for all \(t\ge 1\),
\begin{equation}\label{eq:effective_kappa}
1 + \sum_{k=1}^{t}\ \prod_{\ell=k}^{t}(1-\kappa_\ell)\ \le\ \frac{1}{\kappa}.
\end{equation}
\end{assumption}

Notice the average curvature is never worse than the worst-case curvature: $\kappa \ge \inf_t \kappa_t$.
In particular, a chain with occasional small curvature can still have geometric contraction on average.

In this section, we will also have to assume a finite diameter and Lipschitz observables.

\begin{assumption}[Finite diameter]\label[assumption]{ass:diam}
$\Omega$ has diameter \(D:=\sup_{x,y\in\Omega} d(x,y) < \infty\).
\end{assumption}

\begin{assumption}[Lipschitz observables]\label[assumption]{ass:lip_obs}
There exists \(L>0\) such that
\(
\Lip(F_t)^{\op}\le L \text{ for all }t.
\)
\end{assumption}

We remark it is possible to refine our results to allow for a time-varying Lipschitz parameter, but we choose to assume a uniform Lipschitz parameter for simplicity.

We are now ready to state the main result of this section.

\begin{restatable}{theorem}{ChernoffCurv}\label[theorem]{thm:curv}
Suppose Assumptions \ref{ass:curv}, \ref{ass:diam}, \ref{ass:lip_obs} are satisfied. Let \(v^2:=\frac{192}{\pi^2}\frac{L^2D^2}{\kappa}\).
Then, for all \(\varepsilon\ge 0\),
\[
\Pbbmu\Big(\lmax\Big(\sum_{j=1}^n  (F_j(X_j) - \E F_j(X_j)) \Big)\ge n\varepsilon\Big)
\le m^{\,2-\pi/4}\,
\exp\left(-\frac{n\varepsilon^2}{2v^2}\right).
\]
\end{restatable}

If we have additional control on the  observables' oscillation, we can improve the dependency on the diameter in the variance proxy to logarithmic. For technical reasons, we will  require a uniform lower bound on the curvature, rather than working with the effective curvature of \cref{ass:curv}.

For $F \colon \Omega \to \mathbb{C}^{m \times m}$ we define the
\emph{oscillation} of $F$ as
\(
\osc(F) \coloneqq \sup_{x,y\in\Omega}\|F(x)-F(y)\|_{\op}.
\)
We denote with $\Delta_\op$ the largest oscillation: $\Delta_\op \coloneqq \sup_t \osc(F_t)$. Notice $\Delta_\op \in (0,LD]$.

\begin{assumption}[Uniform curvature] \label[assumption]{ass:unif_kappa}
There exists $\kappa > 0$ such that, for all $t$, $\kappa_t \ge \kappa$.
\end{assumption}

\begin{restatable}{theorem}{ChernoffCurvDiam}\label[theorem]{thm:curv_diam}
Suppose Assumptions \ref{ass:diam}, \ref{ass:lip_obs}, and \ref{ass:unif_kappa}  are satisfied. 
Let $
\overline v^2:=\frac{3200}{\pi^2}\,
\Delta_{\op}^2\kappa^{-1}\left(1+\log\frac{LD}{\Delta_{\op}}\right).
$
Then, for every \(\varepsilon>0\),
\[
\Pbbmu\Big(\lmax\Big(\sum_{j=1}^n (F_j(X_j) - \E F_j(X_j))\Big)\ge n\varepsilon\Big)
\le m^{\,2-\pi/4}\,
\exp\left(-\frac{n\varepsilon^2}{2\overline v^2}\right).
\]
\end{restatable}

The variance proxy in the sub-Gaussian regime is of the correct order, as shown in \cref{app:tight}.
The idea is to construct a (homogeneous) chain that has uniform stationary distribution on the interval $[0,D]$, $\Omega(1)$ curvature, and is supported on a finite set after a finite number of steps. This will make the spectral gap of the chain zero. Furthermore, the chain will forget its initial state in $\Omega(\log D)$ steps, making a dependency on the diameter unavoidable.

%% file: spectral.tex
\subsection{Regime (B): Inhomogeneous spectral gap and bounded observables}
\label{sec:spectral}

In this section we consider a general, potentially unbounded state space $\Omega$.
We do not require $\Omega$ to be Polish nor metric. We consider
 time-inhomogeneous Markov chains that have a positive spectral gap,
  where the spectral gap has a particular definition tailored to the inhomogeneous setting
  and previously studied by \citet{SCZ} (see \cref{def:sigma}).

 Besides this spectral condition, we require the observables to be bounded in the following sense.

\begin{assumption}[Bounded oscillation]\label[assumption]{ass:Delta}
Assume there exist constants $\Delta_{\op},\Delta_{\mathrm F}\in (0,\infty)$ such that, for all $t\ge 1$,
\(
\Delta_{\op} \coloneqq \sup_{t\ge 0,\, x,y\in\Omega}\|F_t(x)-F_t(y)\|_{\op}
\text{ and }
\Delta_{\mathrm F} \coloneqq \sup_{t\ge 0,\, x,y\in\Omega}\|F_t(x)-F_t(y)\|_{\mathrm F}.
\)
\end{assumption}
Notice that $\Delta_{\mathrm F} \le \sqrt{2 \sup_{t,x} \rank F_t(x)} \Delta_{\op}$.
In particular, if $F_t$ is a matrix dilation of some vector, $\rank F_t(x) = 2$ and $\Delta_{\mathrm F} \asymp \Delta_{\op}$.

Given a probability measure $\nu$ on $\Omega$, we define the $L^2(\nu)$ norm as
\(
\|f\|_{2,\nu}:=\Big(\int |f|^2\,d\nu\Big)^{1/2}.
\)
We say $f \in L^2(\nu)$ if $\|f\|_{2,\nu} < \infty$.
We denote with $L^2_0(\nu):=\{f\in L^2(\nu):\nu(f)=0\}$
the restriction of $L^2(\nu)$ to zero-mean functions.

\begin{definition}\label[definition]{def:sigma}
Let \(t\ge 1\). Since \(\mu_t=\mu_{t-1}P_t\), the Markov operator $P_t$
maps \(L^2(\mu_t)\) to \(L^2(\mu_{t-1})\). Define
the \emph{one-step (second) singular value}
\[
\sigma_t
:=\|P_t\|_{L^2_0(\mu_t)\to L^2_0(\mu_{t-1})}
=\sup\Big\{\|P_t f\|_{2,\mu_{t-1}}:\ \|f\|_{2,\mu_t}=1,\ \mu_t(f)=0\Big\}.
\]
Equivalently, for all \(f\in L^2(\mu_t)\) with \(\mu_t(f)=0\),
\begin{equation}\label{eq:L2_contract_scalar}
\|P_t f\|_{2,\mu_{t-1}}\le \sigma_t\,\|f\|_{2,\mu_t}.
\end{equation}
\end{definition}

It can be checked that $\sigma_t \in [0,1]$ (see \cref{prop:L2_contraction_basic}).
Essentially, $1-\sigma_t$ will play the role of the spectral gap of the operator $P_t$.
As done for curvature, we want this spectral gap to be strictly positive ``on average''.

\begin{assumption}[Positive spectral gap]\label[assumption]{ass:lambda}
We assume there exists \(\lambda \in (0,1]\) such that, for every \(t\ge 1\),
\(
\sum_{k=1}^{t+1}\ \prod_{\ell=k}^{t}\sigma_\ell\ \le\ \frac{1}{\lambda}.
\)
\end{assumption}

The main result of this section is as follows.
\begin{restatable}{theorem}{ChernoffSpec}\label{thm:spec}
Suppose Assumptions \ref{ass:Delta} and \ref{ass:lambda} are satisfied.
Let \(v_B^2:=\frac{768}{\pi^2}\,\frac{\Delta_{\op}\Delta_{\mathrm F}}{\lambda}\).
Then, for every \(\varepsilon \ge 0\),
\[
\Pbbmu\Big(\lmax\Big(\sum_{j=1}^n(F_j(X_j) - \E F_j(X_j))\Big)\ge n\varepsilon\Big)
\le m^{\,2-\pi/4}\,
\exp\left(-\frac{n\varepsilon^2}{2v_B^2}\right).
\]
\end{restatable}

Notice \cref{thm:spec} does not require bounded diameter or Lipschitz observables. However,  it is typically hard to obtain a good estimate of $\sigma_t$, since it requires a good control on $\mu_t$ (unless each Markov kernel has a common stationary distribution). See \citep{SCZ} for a discussion on related issues and the connection with the \emph{merging time} of a Markov chain. See also \citep{SCZwaves,SCZsurvey} for examples of inhomogeneous chains for which estimates on $\sigma_t$ are obtainable.

%% file: proofs.tex
\section{Proofs}
\label{sec:proofs_main}

\subsection{Preliminaries results}
This section is devoted to proving the results in \cref{sec:preliminaries}.

\Renewal*
\begin{proof}
First, observe that
\[
 W_n(X_n)^\ast W_n(X_n)  = \exp\!\Big(\frac12 e^{-i\phi}\,s\,\widetilde F_n(X_n)\Big) \exp\!\Big(\frac12 e^{i\phi}\,s\,\widetilde F_n(X_n)\Big)
= \exp\!\Big( \gamma\,s\,\widetilde F_n(X_n)\Big) = \Theta_{n,1}(X_n).
\]
Since
\[
M_nM_n^\ast
=
W_1(X_1)\cdots W_n(X_n)\,W_n(X_n)^\ast\cdots W_1(X_1)^\ast,
\]
we have that
\[
a_n(\phi)=\E\tr\big(M_{n-1}\,\Theta_{n,1}(X_n)\,M_{n-1}^\ast\big).
\]
Using \(\Theta_{n,1}(X_n)=B_{n,1}+H_{n,1}(X_n)\),
\[
a_n(\phi)=\E\tr(M_{n-1}B_{n,1}M_{n-1}^\ast)+\E\tr(M_{n-1}H_{n,1}(X_n)M_{n-1}^\ast).
\]
Conditioning on \(X_{n-1}, \dots, X_1\) we have:
\[
\E\!\left[H_{n,1}(X_n)\mid X_{n-1}, \dots, X_1\right]=(P_nH_{n,1})(X_{n-1}).
\]
Therefore, by the definition of  \(\Theta_{n,2}\),
\begin{align*}
&\E\!\left[M_{n-1}H_{n,1}(X_n)M_{n-1}^\ast\mid X_{n-1},\dots,X_1\right] \\
&\qquad= M_{n-2}\,W_{n-1}\,(X_{n-1}) \, \E\!\left[\,(P_nH_{n,1})(X_{n-1})\,\mid X_{n-1},\dots,X_1\right] \,W_{n-1}(X_{n-1})^\ast\,M_{n-2}^\ast \\
&\qquad= M_{n-2}\,\Theta_{n,2}(X_{n-1})\,M_{n-2}^\ast.
\end{align*}
Taking traces and expectations gives
\[
\E\tr(M_{n-1}H_{n,1}(X_n)M_{n-1}^\ast)=\E\tr(M_{n-2}\Theta_{n,2}(X_{n-1})M_{n-2}^\ast).
\]
Iterating this decomposition yields the identity
\[
a_n(\phi)=\sum_{i=1}^n \E\tr\!\big(M_{n-i}\,B_{n,i}\,M_{n-i}^\ast\big)\;+\;\E\tr\big(H_{n,n}(X_1)\big).
\]
The final remainder vanishes since \(H_{n,n}\) is centered under \(\mu_1\):
by construction \(B_{n,n}=\E[\Theta_{n,n}(X_1)]\), hence \(\E[H_{n,n}(X_1)]=0\) and so \(\E\tr(H_{n,n}(X_1))=0\).
Thus,
\(a_n(\phi)=\sum_{i=1}^n \E\tr(M_{n-i}B_{n,i}M_{n-i}^\ast)\).

Finally, for each \(i\), \(M_{n-i}^\ast M_{n-i}\succeq 0\) and \(B_{n,i}\in\Hm\), so
\[
\tr(M_{n-i}B_{n,i}M_{n-i}^\ast)=\tr(B_{n,i}M_{n-i}^\ast M_{n-i})
\le\|B_{n,i}\|_{\op} \tr(M_{n-i}^\ast M_{n-i})
\]
Taking expectations yields
\[
\E\tr(M_{n-i}B_{n,i}M_{n-i}^\ast)
\le \|B_{n,i}\|_{\op}\,\E\tr(M_{n-i}^\ast M_{n-i})
= \|B_{n,i}\|_{\op}\,a_{n-i}(\phi).
\]
Substituting this into the preceding sum yields \eqref{eq:renewal_ineq}.
\end{proof}

\TraceBound*
\begin{proof}
Fix \(n\ge 1\) and \(s\ge 0\), and condition on the trajectory \((X_1,\dots,X_n)\).
\cref{thm:GLSS} with \(H_j := s\,\widetilde F_j(X_j)\in\Hm\) yields
\[
\tr\exp\!\Bigl(\frac{\pi}{4}\sum_{j=1}^n s\,\widetilde F_j(X_j)\Bigr)
\le
m^{1-\pi/4}\int_{-\pi/2}^{\pi/2}
\tr\Bigl(\prod_{j=1}^n e^{\frac12 e^{i\phi}s\widetilde F_j(X_j)}\cdot
\prod_{j=n}^1 e^{\frac12 e^{-i\phi}s\widetilde F_j(X_j)}\Bigr)\,\nu(\phi).
\]
Taking expectations, 
\[
\E_{\mu_0}\tr\exp\!\Bigl(\frac{\pi}{4}\sum_{j=1}^n s\,\widetilde F_j(X_j)\Bigr)
\le
m^{1-\pi/4}\int_{-\pi/2}^{\pi/2} \E_{\mu_0}
\tr\Bigl(\prod_{j=1}^n e^{\frac12 e^{i\phi}s\widetilde F_j(X_j)}\cdot
\prod_{j=n}^1 e^{\frac12 e^{-i\phi}s\widetilde F_j(X_j)}\Bigr)\,\nu(\phi).
\]
Bounding the integral by the supremum, we obtain
\[
\E_{\mu_0}\tr\exp\!\Bigl(\frac{\pi}{4}\sum_{j=1}^n s\,\widetilde F_j(X_j)\Bigr)
\le
m^{1-\pi/4} \sup_{\phi\in[-\pi/2,\pi/2]} \E_{\mu_0}
\tr\Bigl(\prod_{j=1}^n e^{\frac12 e^{i\phi}s\widetilde F_j(X_j)}\cdot
\prod_{j=n}^1 e^{\frac12 e^{-i\phi}s\widetilde F_j(X_j)}\Bigr).
\]
Recall
\(M_n=\prod_{j=1}^n W_j(X_j)\) with \(W_j(x)=\exp(\frac12 e^{i\phi}s\widetilde F_j(x))\)
and \(W_j(x)^\ast=\exp(\frac12 e^{-i\phi}s\widetilde F_j(x))\).
Therefore,
\begin{equation} \label{eq:glss_reduce}
\E_{\mu_0}\tr\exp\!\Bigl(\frac{\pi}{4} s\sum_{j=1}^n \widetilde F_j(X_j)\Bigr)
\le
m^{1-\pi/4}\sup_{\phi\in[-\pi/2,\pi/2]} \E\tr(M_n M_n^\ast)
=
m^{1-\pi/4}\sup_{\phi\in[-\pi/2,\pi/2]} a_n(\phi).
\end{equation}

Now, from \eqref{eq:renewal_ineq},
\(a_n(\phi)\le (\sum_{i=1}^nb_{n,i}(\phi))\max_{0\le k\le n-1}a_k(\phi)\).
Assumption \eqref{eq:beta_sum} yields
\(a_n(\phi)\le (1+Cs^2)\max_{0\le k\le n-1}a_k(\phi)\).
Since \(a_0(\phi)=\E\tr(I)=m\), we obtain
\(a_n(\phi)\le m\,(1+Cs^2)^n\le m\,\exp(Cns^2)\).
Finally, \cref{eq:mgf_generic} follows by \cref{eq:glss_reduce}.
\end{proof}

\subsection{Regime (A): Proof of \cref{thm:curv}}

We start by showing that a Markov operator $P$ contractive in (scalar) Lipschitz seminorm is also contractive for the operator Lipschitz seminorm. This means we do not have to specialise the standard (scalar) notion of curvature to the high-dimensional setting.
\begin{restatable}{lemma}{ScalarToMatrixLip}\label[lemma]{lem:scalar_to_matrix_Lip}
Let \((\Omega,d)\) be a metric space and let \(P\) be a Markov operator acting on Lipschitz functions.
Assume that there exists \(a\ge 0\) such that, for every real-valued \(h:\Omega\to\mathbb R\),
\[
\Lip(Ph)\le a\,\Lip(h).
\]
Then, for every Hermitian-matrix-valued \(F:\Omega\to\Hm\),
\[
\Lip(PF)^{\op}\le a\,\Lip(F)^{\op}.
\]
\end{restatable}

The proof uses the variational characterisation of the eigenvalues for Hermitian matrices and applies the scalar Lipschitz contractivity of $P$ to the function $h_u(z) \colon \Omega \to \mathbb{R}$, \(h_u(z):=\langle u,F(z)u\rangle\) for some $u \in \mathbb{C}^m$.

\begin{proof}
Fix \(x\neq y\). By the variational characterisation of the operator norm for Hermitian matrices,
\[
\|(PF)(x)-(PF)(y)\|_{\op}
=\sup_{\|u\|_2=1}\Big|\langle u,((PF)(x)-(PF)(y))u\rangle\Big|.
\]
For each unit vector \(u\in\mathbb C^m\), define the scalar function \(h_u(z):=\langle u,F(z)u\rangle\).
Then for all \(z,z'\),
\(
|h_u(z)-h_u(z')|\le \|F(z)-F(z')\|_{\op},
\)
hence \(\Lip(h_u)\le \Lip(F)^{\op}\).
Moreover, by linearity of \(P\),
\(
(Ph_u)(x)=\langle u,(PF)(x)u\rangle
\)
for all \(x\).
Therefore,
\begin{align*}
\Big|\langle u,((PF)(x)-(PF)(y))u\rangle\Big|
&=|(Ph_u)(x)-(Ph_u)(y)|\\
&\le \Lip(Ph_u)\,d(x,y)\\
&\le a\,\Lip(h_u)\,d(x,y)\\
&\le a\,\Lip(F)^{\op}\,d(x,y).
\end{align*}
Taking the supremum over \(\|u\|_2=1\) yields
\(
\|(PF)(x)-(PF)(y)\|_{\op}\le a\,\Lip(F)^{\op}\,d(x,y).
\)
Divide by \(d(x,y)\) and take the supremum over \(x\neq y\).
\end{proof}

Key to the results in this section is the simple observation that centred functions have bounded operator norm.

\begin{restatable}{lemma}{DiamCentre}\label[lemma]{lem:diam_center}
Let \((\Omega,d)\) be a metric space with finite diameter
\(
\diam(\Omega):=\sup_{x,y\in\Omega} d(x,y) \coloneqq D < \infty.
\)
Let \(\mu\) be a probability measure on \(\Omega\), and let \(H:\Omega\to\C^{m\times m}\) be a map
with \(\Lip(H)^{\op}<\infty\).
Let \(\widetilde H(x):=H(x)-\mu(H)\).
Then
\(
\Lip(\widetilde H)^{\op}=\Lip(H)^{\op}
\text{ and }
\|\widetilde H\|_{\infty}\le D\,\Lip(H)^{\op}.
\)
\end{restatable}

\begin{proof}
The Lipschitz seminorm is invariant under addition of constants, so
\(\Lip(\widetilde H)^{\op}=\Lip(H)^{\op}\).

Fix \(x\in\Omega\). Since \(\mu(\widetilde H)=0\),
\[
\widetilde H(x)=\int\big(\widetilde H(x)-\widetilde H(y)\big)\,\mu(dy)
=\int\big(H(x)-H(y)\big)\,\mu(dy).
\]
Hence, by triangle inequality and the definition of \(\Lip(H)^{\op}\),
\[
\|\widetilde H(x)\|_{\op}
\le
\int \|H(x)-H(y)\|_{\op}\,\mu(dy)
\le
\Lip(H)^{\op}\int d(x,y)\,\mu(dy)
\le
D\,\Lip(H)^{\op}.
\]
Taking the supremum over \(x\) gives the claim.
\end{proof}

From the discussion in \cref{sec:preliminaries}, the proof of \cref{thm:curv} essentially reduces to bounding the coefficients $(b_{n,i})$ from \cref{prop:renewal}.
The main idea is to bound the norm of the centred function $H_i$ by iteratively applying the Lipschitz contractivity of the Markov operators. We will need an upper bound on the parameter $s$ to ensure the matrices $W_j(x)$ have small norm and we can therefore recursively bound their product.

To bound the coefficients $(b_{n,i})$ (which depend on product of matrices) we will need a bound on Lipschitzness and operator norm of the function of simple matrices $W_j$.

\begin{lemma}\label[lemma]{lem:W_bounds_curv}
Let \(s\ge 0\),  \(\phi \in [-\pi/2,\pi/2]\), and \(j\ge 1\). Then,
\begin{align*}\label{eq:W_bounds_curv}
\norminf{W_j}&\le \exp\!\Big(\frac{sLD}{2}\Big),
\qquad
\Lip(W_j)^\op\le \frac{s}{2}\exp\!\Big(\frac{sLD}{2}\Big)\,L,
\qquad
\Lip(W_j^\ast)^\op=\Lip(W_j)^\op,\\
\Lip(\Theta_{n,1})^\op &\le s\exp\!\Big(sLD\Big)\,L.
\end{align*}
\end{lemma}
\begin{proof}
Let \(\alpha=\frac12 e^{i\phi}s\), so \(\mathrm{Re}(\alpha)=s\gamma/2\in[0,s/2]\) and \(|\alpha|=s/2\).
By Lemma~\ref{lem:diam_center} and Lemma~\ref{lem:exp_bounds},
\(\|W_j(x)\|_{\op}=\|e^{\alpha\widetilde F_j(x)}\|_{\op}\le e^{|\alpha|\|\widetilde F_j(x)\|_{\op}}\le e^{sLD/2}\).
For Lipschitzness, apply Lemma~\ref{lem:exp_bounds} with \(A=\alpha\widetilde F_j(x)\), \(B=\alpha\widetilde F_j(y)\):
\[
\|W_j(x)-W_j(y)\|_{\op}\le e^{sLD/2}\,|\alpha|\,\|\widetilde F_j(x)-\widetilde F_j(y)\|_{\op}
\le e^{sLD/2}\frac{s}{2}L\,d(x,y).
\]
Lipschitzness of $\Theta_{n,1}$ follows similarly.
\end{proof}

We can now bound the coefficients $(b_{n,i})$.

\begin{restatable}{lemma}{CoeffBoundsCurv}\label[lemma]{lem:coeff_inhomW}
Let $0<s \le 1/(LD)$ and \(\phi\in[-\pi/2,\pi/2]\) with \(\gamma=\cos\phi\).
Let \(b_{n,i}(\phi)\) be defined as in Proposition~\ref{prop:renewal}.
Then,
\begin{equation}\label{eq:beta_i_inhomW}
b_{n,i}(\phi)\ \le
\begin{cases}
\begin{aligned}
&\exp\!\Big(\frac{s^2L^2D^2}{2}\Big) &\quad i=1; \\[1ex]
&2s^2L^2D^2\; \prod_{\ell=n-i+2}^{n} e^{2sLD}\ (1-\kappa_\ell)& i\ge 2.
\end{aligned}
\end{cases}
\end{equation}
\end{restatable}

\begin{proof}
(i) For \(i=1\), we have \(B_{n,1}=\mu_n(\exp(s\gamma\,\widetilde F_n))\).
By Assumption~\ref{ass:lip_obs} and \cref{lem:diam_center},
\(\|\widetilde F_n(x)\|_{\op}\le LD\) for all \(x\), hence
\(-LD\,\Id\preceq \widetilde F_n(X_n)\preceq LD\,\Id\) almost surely under \(X_n\sim\mu_n\).
Apply Lemma~\ref{lem:matrix_hoeffding} with \(R=LD\) to obtain
\[
B_{n,1}=\mu_n\big(e^{s\gamma \widetilde F_n}\big)\ \preceq\ \exp\!\Big(\frac{s^2L^2D^2}{2}\Big)\Id,
\]
hence \(b_{n,1}(\phi)=\|B_{n,1}\|_{\op}\le \exp(s^2L^2D^2/2)\).

(ii) Fix \(i\ge 2\) and set \(\tau:=n-i+1\).
Let \(G:=P_{\tau+1}H_{n,i-1}\), and notice \(\mu_\tau(G)=0\). By definition,
\[
B_{n,i}=\mu_\tau\!\big(W_\tau\,G\,W_\tau^\ast\big).
\]
By writing \(W_\tau=\mu_\tau(W_\tau)+(W_\tau-\mu_\tau(W_\tau))\) and using \(\mu_\tau(\mu_\tau(W_\tau)\, G \, \mu_\tau(W_\tau)^\ast)=0\):
\[
B_{n,i}
=\mu_\tau\!\big((W_\tau-\mu_\tau W_\tau)\,G\,(\mu_\tau W_\tau)^\ast\big)
+\mu_\tau\!\big((\mu_\tau W_\tau)\,G\,(W_\tau^\ast-\mu_\tau W_\tau^\ast)\big)
+\mu_\tau\!\big((W_\tau-\mu_\tau W_\tau)\,G\,(W_\tau^\ast-\mu_\tau W_\tau^\ast)\big).
\]
By the submultiplicativity of the operator norm,
\begin{align}
\|B_{n,i}\|_{\op}
&\le
2\,\norminf{W_\tau}\,\norminf{W_\tau-\mu_\tau W_\tau}\,\mu_\tau(\|G\|_{\op})
+\norminf{W_\tau-\mu_\tau W_\tau}^2\,\mu_\tau(\|G\|_{\op}) \nonumber \\
&\le 4\,\norminf{W_\tau}\,\norminf{W_\tau-\mu_\tau W_\tau}\,\mu_\tau(\|G\|_{\op}),\label{eq:Bi_prefactor_inhomW}
\end{align}
where in the last inequality we have used $\norminf{W_\tau-\mu_\tau W_\tau} \le 2 \norminf{W_\tau}$.
By Lemma~\ref{lem:W_bounds_curv} and Lemma~\ref{lem:diam_center},
\[
\norminf{W_\tau}\le \exp\!\Big(\frac{sLD}{2}\Big)
\quad \text{and} \quad
\norminf{W_\tau-\mu_\tau W_\tau}
\le \sup_{x,y}\|W_\tau(x)-W_\tau(y)\|_{\op}
\le \frac{sLD}{2}\,\exp\!\Big(\frac{sLD}{2}\Big).
\]
Therefore,
\begin{equation}\label{eq:Bi_prefactor_inhomW2}
\|B_{n,i}\|_{\op}\le 2sLD\,e^{sLD} \,\mu_\tau(\|G\|_{\op}).
\end{equation}

Since \(\mu_\tau(G)=0\) and \(\diam(\Omega)=D\), we have that
\begin{equation}\label{eq:sup_by_lip_inhomW}
\norminf{G}\le D\,\Lip(G)^{\op}.
\end{equation}

Moreover, by \cref{ass:curv} and Lemma~\ref{lem:scalar_to_matrix_Lip},
\begin{equation}\label{eq:lipG}
\Lip(G)^{\op}=\Lip(P_{\tau+1}H_{n,i-1})^{\op}\le (1-\kappa_{\tau+1})\,\Lip(H_{n,i-1})^{\op}.
\end{equation}

We next bound \(\Lip(H_{n,i-1})^{\op}\).  We start with \(i=2\).  By
\cref{lem:W_bounds_curv},
\begin{equation} \label{eq:H1}
\Lip(H_{n,1})^{\op}\le \Lip(\Theta_{n,1})^{\op}\le sL\,e^{sLD}.
\end{equation}
%

For \(i\ge3\), since \(H_{n,i-1}=\Theta_{n,i-1}-B_{n,i-1}\) and \(B_{n,i-1}\)
is constant, \(\Lip(H_{n,i-1})^{\op}=\Lip(\Theta_{n,i-1})^{\op}\).  Recalling
the definition of \(\Theta_{n,i-1}\), and using the Lipschitz product rule
(\cref{lem:Lip_product_rule}), we have
\[
\Lip(\Theta_{n,i-1})^{\op}
\le \|W_{\tau+1}\|_\infty^{2}\,\Lip(P_{\tau+2}H_{n,i-2})^{\op}
+2\,\|W_{\tau+1}\|_\infty\,\Lip(W_{\tau+1})^{\op}\,\|P_{\tau+2}H_{n,i-2}\|_\infty.
\]
Moreover,
\(\mu_{\tau+1}(P_{\tau+2}H_{n,i-2})=\mu_{\tau+2}(H_{n,i-2})=0\), hence by
Lemma~\ref{lem:diam_center},
\[
\|P_{\tau+2}H_{n,i-2}\|_\infty
\le D\,\Lip(P_{\tau+2}H_{n,i-2})^{\op}.
\]
By Lemma~\ref{lem:W_bounds_curv}, \(\|W_{\tau+1}\|_\infty\le e^{sLD/2}\) and
\(\Lip(W_{\tau+1})^{\op}\le \frac{sL}{2}e^{sLD/2}\), and so
\[
\Lip(H_{n,i-1})^{\op}
\le e^{sLD}(1+sLD)\,\Lip(P_{\tau+2}H_{n,i-2})^{\op}
\le e^{2sLD}(1-\kappa_{\tau+2})\,\Lip(H_{n,i-2})^{\op}.
\]
Iterating this estimate for \(i\ge3\), and using \eqref{eq:H1} for the base
case, gives
\[
\Lip(H_{n,i-1})^{\op}\le
(e^{2sLD})^{i-2}\,\left(\prod_{\ell=\tau+2}^{n}(1-\kappa_{\ell})\right)\,\Lip(H_{n,1})^{\op},
\]
with the same bound for \(i=2\) with the convention that an empty product is equal to one.
Combining this with \cref{eq:Bi_prefactor_inhomW2,eq:sup_by_lip_inhomW,eq:lipG}  yields
\[
b_{n,i}(\phi)=\|B_{n,i}\|_{\op}
\le 2s^2L^2D^2\,(e^{2sLD})^{i-1}\,\prod_{\ell=\tau+1}^{n} (1-\kappa_{\ell}),
\]
which proves \eqref{eq:beta_i_inhomW}.
\end{proof}

To apply \cref{lem:close} we just need to sum the coefficients $(b_{n,i})$. In particular, we want to upper bound $\sum_{i=2}^n \prod_{\ell=n-i+2}^{n} e^{2sLD}\ (1-\kappa_\ell)$ using \cref{ass:curv}, but notice that each factor in the products is ``tilted'' by $e^{2sLD}$. The next lemma (which is trivial if $\kappa_\ell = \kappa$ for all $\ell$) ensures this tilting factor keeps the products summable.

\begin{restatable}{lemma}{Tilted}\label[lemma]{lem:weighted_renewal}
Assume $0\le \kappa_\ell\le 1$ for all $\ell\ge 1$ and that there exists $\kappa\in(0,1]$ such that
\begin{equation}\label{eq:effective_kappa_repeat}
1+\sum_{k=1}^{t}\ \prod_{\ell=k}^{t}(1-\kappa_\ell)\ \le\ \frac{1}{\kappa}  \qquad \text{ for all } t\ge1.
\end{equation}
Then, for every $n\ge 1$,
\[
\sum_{i=1}^n \Big(e^{\pi\kappa/24}\Big)^{\,i-1}\ \prod_{\ell=n-i+2}^{n}(1-\kappa_\ell)
\ \le\ \frac{3}{\kappa}.
\]
\end{restatable}

\begin{proof}
Let \(u_\ell:=1-\kappa_\ell\in[0,1]\) and \(r:=e^{\pi\kappa/24}\). For \(t\ge 1\), define
\[
a_t:=\sum_{k=1}^{t}\prod_{\ell=k}^{t}u_\ell,
\qquad
b_t:=\sum_{i=1}^{t} r^{\,i-1}\prod_{\ell=t-i+2}^{t}u_\ell,
\]
with the convention that an empty product equals \(1\).
Note that \(b_t\) is exactly the quantity to be bounded.

Let \(t > 1\). Directly from the definitions,
\[
a_t=u_t(1+a_{t-1}),
\qquad
b_t=1+r u_t b_{t-1}.
\]

We prove by induction that for all \(t\ge 1\),
\begin{equation}\label{eq:St_induction}
b_t \le \frac32\,(1+a_t).
\end{equation}
The base case is immediate. Assume \eqref{eq:St_induction} holds at time \(t-1\). Then
\[
b_t=1+r u_t b_{t-1}
\le 1+\frac32 r u_t(1+a_{t-1})
=1+\frac32 r a_t.
\]
To close the induction, it suffices to show
\[
1+\frac32 r a_t\le \frac32(1+a_t),
\]
or equivalently \((r-1)a_t\le 1/3\). By \eqref{eq:effective_kappa_repeat}, \(a_t\le1/\kappa\). Moreover, since \(\kappa\in(0,1]\), the map \(x\mapsto (e^{\pi x/24}-1)/x\) is increasing on \((0,\infty)\), and \(\pi/24<\log(4/3)\), we have
\[
(r-1)a_t
\le \frac{e^{\pi\kappa/24}-1}{\kappa}
\le e^{\pi/24}-1
<\frac13.
\]
This proves \eqref{eq:St_induction}. Applying it with \(t=n\), and using again \eqref{eq:effective_kappa_repeat} and \(\kappa\le1\), gives
\[
b_n \le \frac32(1+a_n)\le \frac32\Big(1+\frac{1}{\kappa}\Big)\le \frac{3}{\kappa}.
\]
\end{proof}

\cref{thm:curv} is proved by applying the standard matrix Chernoff bound argument \eqref{eq:laplace_maxeig} together with \cref{thm:GLSS} and \cref{lem:close}, with the parameter $C$ in \cref{lem:close} derived from \cref{lem:coeff_inhomW,lem:weighted_renewal}.

\ChernoffCurv*
\begin{proof}
Set \(\overline r_{\max}:=\frac{\pi^2\kappa}{192LD}\). Fix \(r\in[0,\overline r_{\max}]\), \(\phi \in [-\pi/2,\pi/2]\), and let \(s \coloneqq \frac{4}{\pi}r\).
Then \(2sLD\le \pi\kappa/24\). By \cref{lem:coeff_inhomW,lem:weighted_renewal},
\begin{align*}
\sum_{i=1}^n b_{n,i}(\phi)
&\le e^{(1/2)s^2L^2D^2} + 2s^2L^2D^2 \sum_{i=2}^n \;\big(e^{2sLD}\big)^{\,i-1}\ \prod_{\ell=n-i+2}^{n}(1-\kappa_\ell) \\
&\le 1 + 2s^2L^2D^2 \sum_{i=1}^n \;\big(e^{2sLD}\big)^{\,i-1}\ \prod_{\ell=n-i+2}^{n}(1-\kappa_\ell) \\
&\le 1 + 2s^2L^2D^2 \sum_{i=1}^n \;\big(e^{\pi\kappa/24}\big)^{\,i-1}\ \prod_{\ell=n-i+2}^{n}(1-\kappa_\ell) \\
&\le 1 + 6s^2L^2D^2 \kappa^{-1}.
\end{align*}
In the second line we used \(e^{x/2}\le1+2x\) for \(0\le x\le1\), with \(x=s^2L^2D^2\le(\pi/48)^2<1\).

Apply Lemma~\ref{lem:close} with \(C = 6L^2D^2\kappa^{-1}\). Since \(s=(4/\pi)r\),
\[
\E_{\mu_0}\tr\exp\!\Bigl(r\sum_{j=1}^n \widetilde F_j(X_j)\Bigr)
\le m^{2-\pi/4}\exp\!\Bigl(\frac{16C}{\pi^2}n r^2\Bigr)
= m^{2-\pi/4}\exp\!\Bigl(\frac12 n v^2 r^2\Bigr),
\]
where \(v^2=\frac{32C}{\pi^2}=\frac{192}{\pi^2}\frac{L^2D^2}{\kappa}\).
Thus, for every \(r\in[0,\overline r_{\max}]\),
\[
\Pbbmu\!\bigl(\lmax(\sum_{j=1}^n \widetilde F_j(X_j))\ge n\varepsilon\bigr)
\le
m^{2-\pi/4}\exp\!\Bigl(-n\bigl(r\varepsilon-\tfrac12 v^2r^2\bigr)\Bigr).
\]

If \(0\le\varepsilon\le LD\), then \(r^\star:=\varepsilon/v^2\le \overline r_{\max}\), because \(v^2\overline r_{\max}=LD\). Choosing \(r=r^\star\) gives
\[
\Pbbmu\!\bigl(\lmax(\sum_{j=1}^n \widetilde F_j(X_j))\ge n\varepsilon\bigr)
\le m^{2-\pi/4}\exp\!\left(-\frac{n\varepsilon^2}{2v^2}\right).
\]
If \(\varepsilon>LD\), then the bound is obvious: by Lemma~\ref{lem:diam_center},
\[
\left\|F_j(X_j)-\E F_j(X_j)\right\|_{\op}\le LD
\]
for each \(j\), and hence
\[
\lmax\left(\sum_{j=1}^n \widetilde F_j(X_j)\right)
\le \sum_{j=1}^n \left\|\widetilde F_j(X_j)\right\|_{\op}
\le nLD<n\varepsilon.
\]
This proves the stated bound for all \(\varepsilon\ge0\).
\end{proof}

\subsection{Regime (A): Improved dependence on the diameter}
Before proving \cref{thm:curv_diam}, we show how the logarithmic dependency arises in a simpler scalar setting.
Let \(P\) be a time-homogeneous Markov kernel with invariant distribution \(\pi\), and assume that
\[
W_1(P(x,\cdot),P(y,\cdot))\le (1-\kappa)d(x,y)
\qquad\text{for all }x,y\in\Omega
\]
for some \(0<\kappa\le1\).
Let \(f:\Omega\to\R\) satisfy \(\pi(f)=0\),
\[
\Lip(f)\le L,\qquad \osc(f)\le \Delta\le LD,
\]
and let the chain be started from stationarity.
We claim that its asymptotic variance is bounded at the same scale as the variance proxy in \cref{thm:curv_diam}:
\[
\sigma_{\mathrm{as}}^2(f)
\coloneqq
\lim_{n\to\infty}\frac1n\operatorname{Var}_{\pi}\!\left(\sum_{t=1}^n f(X_t)\right)
\le
C\,\Delta^2\kappa^{-1}\left(1+\log\frac{LD}{\Delta}\right)
\]
for some constant \(C > 0\).

For every \(n\), stationarity implies the covariance decomposition
\[
\frac1n\operatorname{Var}_{\pi}\!\left(\sum_{t=1}^n f(X_t)\right)
=
\operatorname{Var}_{\pi}(f)
+2\sum_{i=1}^{n-1}\left(1-\frac{i}{n}\right)
\operatorname{Cov}_{\pi}\big(f(X_0),f(X_i)\big).
\]
It is therefore enough to control the covariance between $f(X_0)$ and $f(X_i)$. The key observation is that, for small $i$, it is more convenient to bound this covariance by the oscillation of $f$ rather than curvature.
To this end, notice that since the chain is stationary and \(\pi(f)=0\),
\[
\operatorname{Cov}_{\pi}\big(f(X_0),f(X_i)\big)
=\int f(x)P^if(x)\,\pi(dx).
\]
Since \(\pi(f)=0\), we have \(\norminf{f}\le \osc(f)\le\Delta\).
Moreover \(\pi(P^if)=0\) for every \(i\), and hence
\[
\norminf{P^if}\le \osc(P^if)\le \osc(f)\le \Delta.
\]
On the other hand, by curvature,
\[
\Lip(P^if)\le (1-\kappa)^i\Lip(f)\le L(1-\kappa)^i,
\]
and Lemma~\ref{lem:diam_center} yields
\[
\norminf{P^if}\le D\,\Lip(P^if)\le LD(1-\kappa)^i.
\]
Combining the oscillation bound with the Lipschitz-contraction bound,
\[
\norminf{P^if}\le \Delta\min\left\{1,\frac{LD}{\Delta}(1-\kappa)^i\right\}.
\]
Consequently,
\[
\left|\int f(x)P^if(x)\,\pi(dx)\right|
\le
\Delta^2\min\left\{1,\frac{LD}{\Delta}(1-\kappa)^i\right\}.
\]

It remains to sum this profile for all $i$. Since \(LD/\Delta\ge1\),
when \(\kappa=1\), the next estimate is immediate, so assume \(0<\kappa<1\) and define
\[
i_\star:=1+\left\lceil\frac{\log(LD/\Delta)}{-\log(1-\kappa)}\right\rceil.
\]
Then \(i_\star\le 2+\kappa^{-1}\log(LD/\Delta)\), since \(-\log(1-\kappa)\ge\kappa\), and
\[
\sum_{i\ge1}\min\left\{1,\frac{LD}{\Delta}(1-\kappa)^i\right\}
\le
i_\star+\frac{LD}{\Delta}\sum_{i\ge i_\star}(1-\kappa)^i
\le
2+\kappa^{-1}\log(LD/\Delta)+\kappa^{-1}.
\]
Therefore
\[
\frac1n\operatorname{Var}_{\pi}\!\left(\sum_{t=1}^n f(X_t)\right)
\le
\Delta^2\left(1+2\sum_{i\ge1}\min\left\{1,\frac{LD}{\Delta}(1-\kappa)^i\right\}\right)
\le
7\Delta^2\kappa^{-1}(1+\log(LD/\Delta)).
\]
Letting \(n\to\infty\) yields the displayed bound on \(\sigma_{\mathrm{as}}^2(f)\).

We exploit the intuition outlined above to improve the variance proxy of \cref{thm:curv}. However, a crude application would achieve the correct logarithmic dependency on the diameter in the variance proxy only over a small sub-Gaussian window. To enlarge this window, we must keep track of how curvature contracts the relevant distances below the diameter scale.
In particular, we will need the following way of recording increments at different distance scales.
For a matrix-valued function \(H:\Omega\to\mathbb C^{m\times m}\), we say that an increasing concave function
\(\psi:[0,D]\to[0,\infty)\) controls the increments of \(H\) if
\[
\|H(x)-H(y)\|_{\op}\le \psi(d(x,y))\qquad\text{for all }x,y\in\Omega.
\]

The next lemma shows how applying a positively curved operator improves the control on the increments.

\begin{lemma}\label[lemma]{lem:concave_profile_contract}
Let \(P_t\) have curvature at least \(\kappa\).
If an increasing concave function \(\psi:[0,D]\to[0,\infty)\) controls the increments of \(H\), then
\[
\|P_tH(x)-P_tH(y)\|_{\op}\le \psi((1-\kappa)d(x,y))
\qquad\text{for all }x,y\in\Omega.
\]
\end{lemma}
\begin{proof}
Fix \(x,y\in\Omega\), and let \((U,V)\) be an optimal coupling of \(P_t(x,\cdot)\) and \(P_t(y,\cdot)\) for the \(W_1\) distance.
Then, by Jensen's inequality and the concavity of \(\psi\),
\begin{align*}
\|P_tH(x)-P_tH(y)\|_{\op}
&=\left\|\E\big[H(U)-H(V)\big]\right\|_{\op}\\
&\le \E\|H(U)-H(V)\|_{\op}\\
&\le \E\psi(d(U,V))\\
&\le \psi\big(\E d(U,V)\big).
\end{align*}
By the definition of curvature in \cref{def:curvature} and Assumption~\ref{ass:unif_kappa},
\[
\E d(U,V)=W_1(P_t(x,\cdot),P_t(y,\cdot))\le (1-\kappa)d(x,y).
\]
Since \(\psi\) is increasing, the claim follows.
\end{proof}

The next lemma show we can control the increments of \(H_{n,j}\), as defined in \cref{prop:renewal}, by a certain recursively defined family of functions.

\begin{lemma}\label[lemma]{lem:scale_profile_remainders}
Fix \(n\ge1\), \(s\ge0\), and \(\phi\in[-\pi/2,\pi/2]\).
Let \(H_{n,j}\) be the remainders defined in \cref{prop:renewal}. Define \(\psi_1,\psi_2,\dots\) recursively by
\[
\psi_1(r):=s\Delta_{\op} e^{s\Delta_{\op}}\min\left\{1,\frac{Lr}{\Delta_{\op}}\right\},
\]
and, for \(j\ge2\),
\[
\psi_j(r):=e^{s\Delta_{\op}}\psi_{j-1}((1-\kappa)r)
+s\Delta_{\op} e^{s\Delta_{\op}}\min\left\{1,\frac{Lr}{\Delta_{\op}}\right\}\psi_{j-1}((1-\kappa)D).
\]
Then, \(\psi_j\) $(1\le j \le n)$ is increasing and concave, and \(\psi_j\) controls the increments of \(H_{n,j}\).
\end{lemma}
\begin{proof}
The function \(r\mapsto \min\{1,Lr/\Delta_{\op}\}\) is increasing and concave on \([0,D]\), and it vanishes at \(0\).
By induction, if \(\psi_{j-1}\) is increasing and concave, then
\(r\mapsto e^{s\Delta_{\op}}\psi_{j-1}((1-\kappa)r)\) is increasing and concave.
The second summand in the definition of \(\psi_j\) is a nonnegative constant,
\(s\Delta_{\op} e^{s\Delta_{\op}}\psi_{j-1}((1-\kappa)D)\), times
\(r\mapsto \min\{1,Lr/\Delta_{\op}\}\), and is therefore also increasing and concave.
Since sums of increasing concave functions are increasing and concave, the same is true of each \(\psi_j\).

For later use, we note that, for every \(t\) and \(x\),
\[
\|\widetilde F_t(x)\|_{\op}
=\left\|\int\big(F_t(x)-F_t(y)\big)\,\mu_t(dy)\right\|_{\op}
\le \Delta_{\op}.
\]

We first prove the claim for \(j=1\).
Recall from \cref{prop:renewal} that
\[
H_{n,1}(x)=\exp\big(s\gamma \widetilde F_n(x)\big)-B_{n,1},
\qquad \gamma=\cos\phi\in[0,1].
\]
The constant \(B_{n,1}\) does not affect increments. Moreover, by Assumption~\ref{ass:lip_obs},
\[
\|\widetilde F_n(x)-\widetilde F_n(y)\|_{\op}
=\|F_n(x)-F_n(y)\|_{\op}
\le \Delta_{\op}\min\left\{1,\frac{Ld(x,y)}{\Delta_{\op}}\right\}.
\]
Lemma~\ref{lem:exp_bounds} therefore gives
\[
\|H_{n,1}(x)-H_{n,1}(y)\|_{\op}
\le s e^{s\Delta_{\op}} \|\widetilde F_n(x)-\widetilde F_n(y)\|_{\op}
\le s\Delta_{\op}e^{s\Delta_{\op}}\min\left\{1,\frac{Ld(x,y)}{\Delta_{\op}}\right\}.
\]
This proves the base case.

Assume now that \(\psi_{j-1}\) controls the increments of \(H_{n,j-1}\), and let \(\tau:=n-j+1\).
Set
\[
G:=P_{\tau+1}H_{n,j-1}.
\]
By \cref{lem:concave_profile_contract}, the increments of \(G\) are controlled by \(r\mapsto \psi_{j-1}((1-\kappa)r)\).
Moreover,
\[
\mu_\tau(G)=\mu_\tau(P_{\tau+1}H_{n,j-1})=\mu_{\tau+1}(H_{n,j-1})=0,
\]
where the last equality follows from the centering in \cref{prop:renewal}. Hence, for all \(x\in\Omega\),
\[
\|G(x)\|_{\op}
=\left\|\int \big(G(x)-G(y)\big)\,\mu_\tau(dy)\right\|_{\op}
\le \psi_{j-1}((1-\kappa)D).
\]

Since \(\|\widetilde F_\tau\|_\infty\le\Delta_{\op}\), Lemma~\ref{lem:exp_bounds} gives
\[
\|W_\tau\|_\infty\le e^{s\Delta_{\op}/2},
\]
and, using again Assumption~\ref{ass:lip_obs},
\[
\|W_\tau(x)-W_\tau(y)\|_{\op}
\le \frac{s}{2}e^{s\Delta_{\op}/2}\|\widetilde F_\tau(x)-\widetilde F_\tau(y)\|_{\op}
\le \frac{s\Delta_{\op}}{2}e^{s\Delta_{\op}/2}\min\left\{1,\frac{Ld(x,y)}{\Delta_{\op}}\right\}.
\]
The same bound holds for \(W_\tau^\ast\). Recalling from \cref{prop:renewal} that
\[
\Theta_{n,j}(x)=W_\tau(x)G(x)W_\tau(x)^\ast,\qquad H_{n,j}(x)=\Theta_{n,j}(x)-B_{n,j},
\]
we obtain, for all \(x,y\in\Omega\),
\begin{align*}
\|H_{n,j}(x)-H_{n,j}(y)\|_{\op}
&=\|\Theta_{n,j}(x)-\Theta_{n,j}(y)\|_{\op}\\
&\le \|W_\tau\|_\infty^2\|G(x)-G(y)\|_{\op}
 +2\|W_\tau\|_\infty\|W_\tau(x)-W_\tau(y)\|_{\op}\|G\|_\infty\\
&\le e^{s\Delta_{\op}}\psi_{j-1}((1-\kappa)d(x,y))
+s\Delta_{\op}e^{s\Delta_{\op}}\min\left\{1,\frac{Ld(x,y)}{\Delta_{\op}}\right\}\psi_{j-1}((1-\kappa)D)\\
&=\psi_j(d(x,y)).
\end{align*}
This closes the induction.
\end{proof}

We can now exploit the control of the increments provided by the previous lemma to bound the sum of the coefficients
$\{b_{n,i}\}$.

\begin{lemma}\label[lemma]{lem:sharp_window_sum_bi}
Let $0\le s\le \frac{\kappa}{20\Delta_{\op}(1+\log (LD/\Delta_{\op}))}$.
Then, for all \(n\ge1\) and all \(\phi\in[-\pi/2,\pi/2]\),
\[
\sum_{i=1}^n b_{n,i}(\phi)
\le
1+100\,s^2\Delta_{\op}^2\,\kappa^{-1}(1+\log (LD/\Delta_{\op})).
\]
\end{lemma}
\begin{proof}
Let \(\psi_j\) be as in \cref{lem:scale_profile_remainders}, and, for \(j,k\ge1\), define
\[
q_{j,k}:=\psi_j((1-\kappa)^kD).
\]
From the recursive definition of \(\psi_j\),
\[
q_{j,k}=e^{s\Delta_{\op}} q_{j-1,k+1}
+s\Delta_{\op}e^{s\Delta_{\op}}\min\left\{1,\frac{LD}{\Delta_{\op}}(1-\kappa)^k\right\}q_{j-1,1}
\qquad (j\ge2,\ k\ge1),
\]
while
\[
q_{1,k}=s\Delta_{\op}e^{s\Delta_{\op}}\min\left\{1,\frac{LD}{\Delta_{\op}}(1-\kappa)^k\right\}.
\]
Iterating this identity gives, for every \(j\ge1\),
\begin{equation}\label{eq:qj1_renewal}
q_{j,1}
=
s\Delta_{\op} e^{s\Delta_{\op}j}\min\left\{1,\frac{LD}{\Delta_{\op}}(1-\kappa)^j\right\}
+s\Delta_{\op}\sum_{\ell=1}^{j-1}e^{s\Delta_{\op}\ell}\min\left\{1,\frac{LD}{\Delta_{\op}}(1-\kappa)^\ell\right\}q_{j-\ell,1}.
\end{equation}

We next bound
\[
S:=\sum_{\ell\ge1}e^{s\Delta_{\op}\ell}\min\left\{1,\frac{LD}{\Delta_{\op}}(1-\kappa)^\ell\right\}.
\]
Let \(J:=\left\lceil \kappa^{-1}\log (LD/\Delta_{\op})\right\rceil\). Since \(1-\kappa\le e^{-\kappa}\),
\[
\min\left\{1,\frac{LD}{\Delta_{\op}}(1-\kappa)^\ell\right\}
\le \min\left\{1,\frac{LD}{\Delta_{\op}}e^{-\kappa\ell}\right\}.
\]
Moreover, using \(s\Delta_{\op}\le \kappa/(20(1+\log (LD/\Delta_{\op})))\), \(\kappa\le1\), and \(J\le \kappa^{-1}\log (LD/\Delta_{\op})+1\), we have \(s\Delta_{\op}J\le 1/20\). Therefore,
\[
\sum_{\ell=1}^{J}e^{s\Delta_{\op}\ell}\min\left\{1,\frac{LD}{\Delta_{\op}}(1-\kappa)^\ell\right\}
\le J e^{s\Delta_{\op}J}
\le 2\,\frac{1+\log (LD/\Delta_{\op})}{\kappa}.
\]
For the tail, since \(s\Delta_{\op}\le \kappa/2\),
\[
\sum_{\ell>J}e^{s\Delta_{\op}\ell}\min\left\{1,\frac{LD}{\Delta_{\op}}(1-\kappa)^\ell\right\}
\le \frac{LD}{\Delta_{\op}}\sum_{\ell>J}e^{-(\kappa-s\Delta_{\op})\ell}
\le \frac{(LD/\Delta_{\op}) e^{-(\kappa-s\Delta_{\op})J}}{1-e^{-(\kappa-s\Delta_{\op})}}
\le \frac{2}{1-e^{-(\kappa-s\Delta_{\op})}}.
\]
Since \(0<\kappa-s\Delta_{\op}\le1\), we have \(1-e^{-(\kappa-s\Delta_{\op})}\ge(\kappa-s\Delta_{\op})/2\ge\kappa/4\). Hence,
\[
\sum_{\ell>J}e^{s\Delta_{\op}\ell}\min\left\{1,\frac{LD}{\Delta_{\op}}(1-\kappa)^\ell\right\}\le \frac{8}{\kappa}.
\]
Combining the two estimates,
\begin{equation}\label{eq:S_bound}
S\le 10\,\frac{1+\log (LD/\Delta_{\op})}{\kappa}.
\end{equation}
In particular \(s\Delta_{\op}S\le1/2\).

Summing \eqref{eq:qj1_renewal} over \(1\le j\le N\), we obtain
\[
\sum_{j=1}^N q_{j,1}
\le s\Delta_{\op}S+s\Delta_{\op}S\sum_{j=1}^N q_{j,1}.
\]
Since \(s\Delta_{\op}S\le1/2\),
\[
\left(1-s\Delta_{\op}S\right)\sum_{j=1}^N q_{j,1}
\le s\Delta_{\op}S,
\]
and hence
\begin{equation}\label{eq:sum_qj1_bound}
\sum_{j=1}^N q_{j,1}
\le 2s\Delta_{\op}S
\le 20s\Delta_{\op}\,\frac{1+\log (LD/\Delta_{\op})}{\kappa}.
\end{equation}

We now return to the coefficients \(b_{n,i}(\phi)\). For \(i=1\), since
\(\E_{\mu_n}\widetilde F_n(X_n)=0\) and \(\|\widetilde F_n\|_\infty\le\Delta_{\op}\), Lemma~\ref{lem:matrix_hoeffding} gives
\[
b_{n,1}(\phi)\le \exp(s^2\Delta_{\op}^2/2)\le 1+s^2\Delta_{\op}^2,
\]
where the last inequality uses \(s\Delta_{\op}\le1\).
For \(i\ge2\), set \(\tau:=n-i+1\) and \(G:=P_{\tau+1}H_{n,i-1}\).
By \cref{lem:scale_profile_remainders} and \cref{lem:concave_profile_contract}, for all \(x,y\in\Omega\),
\[
\|G(x)-G(y)\|_{\op}
\le \psi_{i-1}((1-\kappa)d(x,y)).
\]
Moreover,
\[
\mu_\tau(G)=\mu_\tau(P_{\tau+1}H_{n,i-1})=\mu_{\tau+1}(H_{n,i-1})=0,
\]
where the last equality follows from the centering in \cref{prop:renewal}. Hence, for every \(x\in\Omega\),
\[
\|G(x)\|_{\op}
=\left\|\int \big(G(x)-G(y)\big)\,\mu_\tau(dy)\right\|_{\op}
\le \psi_{i-1}((1-\kappa)D)=q_{i-1,1}.
\]
Therefore
\[
\|G\|_\infty\le q_{i-1,1}.
\]
Using \eqref{eq:Bi_prefactor_inhomW}, together with
\[
\|W_\tau\|_\infty\le e^{s\Delta_{\op}/2},
\qquad
\|W_\tau-\mu_\tau W_\tau\|_\infty
\le \sup_{x,y}\|W_\tau(x)-W_\tau(y)\|_{\op}
\le \frac{s\Delta_{\op}}{2}e^{s\Delta_{\op}/2},
\]
we obtain
\[
b_{n,i}(\phi)\le 2s\Delta_{\op}e^{s\Delta_{\op}} q_{i-1,1}.
\]
Since \(s\Delta_{\op}\le1/20\), \(e^{s\Delta_{\op}}\le2\), and \eqref{eq:sum_qj1_bound} holds with \(N=n-1\),
\[
\sum_{i=2}^n b_{n,i}(\phi)
\le 4s\Delta_{\op}\sum_{j=1}^{n-1}q_{j,1}
\le 80s^2\Delta_{\op}^2\,\frac{1+\log (LD/\Delta_{\op})}{\kappa}.
\]
Finally, since \((1+\log (LD/\Delta_{\op}))/\kappa\ge1\),
\[
\sum_{i=1}^n b_{n,i}(\phi)
\le 1+81s^2\Delta_{\op}^2\,\frac{1+\log (LD/\Delta_{\op})}{\kappa}
\le 1+100s^2\Delta_{\op}^2\,\frac{1+\log (LD/\Delta_{\op})}{\kappa}.
\]
\end{proof}

We now have all the ingredients to prove the main result of this section.
\ChernoffCurvDiam*
\begin{proof}
Set
\[
\overline r_{\max}:=\frac{\pi\kappa}{80\,\Delta_{\op}\,(1+\log (LD/\Delta_{\op}))}.
\]
Fix \(r\in[0,\overline r_{\max}]\), \(\phi\in[-\pi/2,\pi/2]\), and set \(s:=4r/\pi\).
Then
\[
s\Delta_{\op}\le \frac{\kappa}{20(1+\log (LD/\Delta_{\op}))},
\]
so \cref{lem:sharp_window_sum_bi} applies. By \cref{lem:close},
\[
\E_{\mu_0}\tr\exp\!\Bigl(r\sum_{j=1}^n \widetilde F_j(X_j)\Bigr)
\le
m^{2-\pi/4}
\exp\!\left(100n s^2\Delta_{\op}^2\kappa^{-1}(1+\log (LD/\Delta_{\op}))\right).
\]
Since \(s=4r/\pi\), this is
\[
\E_{\mu_0}\tr\exp\!\Bigl(r\sum_{j=1}^n \widetilde F_j(X_j)\Bigr)
\le
m^{2-\pi/4}
\exp\!\left(\frac12 n\overline v^2 r^2\right),
\]
with \(\overline v^2\) as in the statement. Let \(\widetilde S_n:=\sum_{j=1}^n \widetilde F_j(X_j)\). By \cref{eq:laplace_maxeig}, for every \(r\in[0,\overline r_{\max}]\),
\[
\Pbbmu\bigl(\lmax(\widetilde S_n)\ge n\varepsilon\bigr)
\le
m^{2-\pi/4}\exp\!\left(-n\bigl(r\varepsilon-\tfrac12\overline v^2r^2\bigr)\right).
\]

For \(0\le\varepsilon\le\Delta_{\op}\), the unconstrained optimizer
\[
r^\star:=\frac{\varepsilon}{\overline v^2}
\]
belongs to the admissible interval. Indeed,
\[
\overline v^2\,\overline r_{\max}
=\frac{40}{\pi}\,\Delta_{\op}
\ge \Delta_{\op}.
\]
Choosing \(r=r^\star\) yields
\[
\Pbbmu\bigl(\lmax(\widetilde S_n)\ge n\varepsilon\bigr)
\le m^{2-\pi/4}\exp\!\left(-\frac{n\varepsilon^2}{2\overline v^2}\right).
\]

Finally, for every \(j\),
\[
\|\widetilde F_j(x)\|_{\op}
=\|F_j(x)-\mu_j(F_j)\|_{\op}
\le \int \|F_j(x)-F_j(y)\|_{\op}\,\mu_j(dy)
\le \Delta_{\op}.
\]
Therefore
\[
\lmax(\widetilde S_n)
\le \sum_{j=1}^n \|\widetilde F_j(X_j)\|_{\op}
\le n\Delta_{\op},
\]
so the event is empty when \(\varepsilon>\Delta_{\op}\).
\end{proof}

\subsubsection{Tightness of the sub-Gaussian regime in \cref{thm:curv_diam}}
\label{app:tight}
We construct an example of a Markov chain with a variance proxy that matches the one of \cref{thm:curv_diam} up to constants.
This chain will be homogeneous but we will consider time-inhomogeneous observables (this will simplify the construction of such an example).

Let $\Omega=[0,D]$ with metric $d(x,y)=|x-y|$ (so $\mathrm{diam}(\Omega)=D$).
Consider the following chain:
\begin{equation}\label{eq:chain}
X_{t+1} \;=\; \frac{X_t}{2} \;+\; \frac{D}{2}\,B_{t+1},
\qquad
 B_{t+1}\sim \mathrm{Bernoulli}\!\left(\tfrac12\right)\ \text{i.i.d.}
\end{equation}
Let $P$ be the Markov operator associated with the chain.

It is straightforward to show this chain has $\Omega(1)$ curvature.
Indeed, for any $x,y\in\Omega$, couple one step of the evolution of the chain by using the same Bernoulli r.v. in \eqref{eq:chain}.
Then,
\[
W_1(P(x,\cdot),P(y,\cdot)) \le \frac{|x-y|}{2},
\]
and therefore $\kappa \ge 1/2$.

Let $\ell\ge 1$ be the largest integer such that
\begin{equation}\label{eq:ell_def}
\pi\,\Delta\,2^{\ell} \le LD.
\end{equation}
For $t=1,2,\dots,\ell$, define the (scalar) observable
\begin{equation}\label{eq:Ft}
F_t(x) \;:=\; \frac{\Delta}{2}\cos\!\Bigl(2\pi\,2^{t}\,\frac{x}{D}\Bigr).
\end{equation}
Then $\osc(F_t) \le \Delta$, and
\[
\Lip(F_t)
\le \sup_x \left|\frac{d}{dx}F_t(x)\right|
= \frac{\Delta}{2}\cdot 2\pi\,2^{t}\cdot \frac{1}{D}
 \le \frac{\pi\Delta\,2^{\ell}}{D}
\le L,
\]
where the last inequality is exactly \eqref{eq:ell_def}. Thus, the hypotheses $\Lip(F_t)\le L$ and $\osc(F_t)\le \Delta$ for all $t\le \ell$ are satisfied.

Let the chain start from the uniform distribution on $\Omega$, which is the stationary distribution of the chain:
\begin{equation}\label{eq:init}
X_0 \sim \mu_0 \coloneqq \mathrm{Unif}([0,D]).
\end{equation}
Unrolling \eqref{eq:chain} gives
\[
X_t=\frac{X_0}{2^t}+\frac{D}{2}\sum_{i=1}^t \frac{B_i}{2^{t-i}},
\]
and, hence,
\[
\frac{2^tX_t}{D}=\frac{X_0}{D}+ z,
\]
for some integer $z$.
Therefore, by \eqref{eq:Ft}, for every $t\le \ell$,
\[
F_t(X_t)
= \frac{\Delta}{2}\cos\!\Bigl(2\pi\,\frac{X_0}{D}\Bigr).
\]
In other words, $F_t(X_t)$ depends only on the starting point.
Moreover, $\E[F_t(X_t)]=0$ for all $t\le \ell$ since $X_0/D$ is uniform on $[0,1]$.

Define
\[
S_\ell := \sum_{t=1}^\ell \bigl(F_t(X_t)-\E[F_t(X_t)]\bigr)=\sum_{t=1}^\ell F_t(X_t).
\]
By the computation above,
\[
S_\ell= \frac{\Delta}{2}\,\ell\,\cos\!\Bigl(2\pi\,\frac{X_0}{D}\Bigr).
\]

Now suppose we want a sub-Gaussian bound of the form
\begin{equation}\label{eq:subg}
\Pbb(S_\ell \ge \ell\varepsilon)\ \le\ \exp\!\left(-\frac{\ell\varepsilon^2}{2v^2}\right)
\qquad\text{for all }\varepsilon\ge 0.
\end{equation}
Choose $\varepsilon = \Delta/4$. Since $X_0/D$ is uniform on $[0,1]$, we have that
\[
\Pbb\left(S_\ell >  \frac{\Delta}{4} \ell\right)=\Pbb\left(\cos\left( 2 \pi X_0/D \right) \ge \frac{1}{2}\right) = \frac{1}{3}.
\]
Therefore,
\[
\frac{1}{3} \le \exp\!\left(-\frac{\ell\varepsilon^2}{2v^2}\right)
\qquad\Rightarrow\qquad
v^2 \ge \frac{\ell\varepsilon^2}{2\log 3}.
\]
By \eqref{eq:ell_def}, $\ell=\Theta(\log(LD/\Delta))$,
which implies
\[
v^2 = \Omega\!\left(\Delta^2 \log\frac{LD}{\Delta}\right)
= \Omega\!\left(\frac{\Delta^2}{\kappa}\log\frac{LD}{\Delta}\right).
\]
This matches (up to constants) the $\Delta^2\kappa^{-1}\log(LD/\Delta)$ variance proxy  of \cref{thm:curv_diam}.

\subsection{Regime (B)}
\label{app:proofs_spec}
This section is devoted to proving \cref{thm:spec}. The proof follows a similar outline as the proof of \cref{thm:curv}, with the inhomogeneous spectral gap $\lambda$ replacing the role of the average curvature $\kappa$.

We first show that $P_t$ is non-expansive in $L^2(\mu_t)\to L^2(\mu_{t-1})$.

\begin{proposition}\label[proposition]{prop:L2_contraction_basic}
For each \(t\ge 1\), \(\|P_t\|_{L^2(\mu_t)\to L^2(\mu_{t-1})}\le 1\). In particular, \(\sigma_t\le 1\).
\end{proposition}
\begin{proof}
Fix \(f\in L^2(\mu_t)\). By Jensen's inequality,
\[
|P_t f(x)|^2=\Big|\int f(y)\,P_t(x,dy)\Big|^2\le \int |f(y)|^2\,P_t(x,dy)=P_t(|f|^2)(x).
\]
Integrate both sides against \(\mu_{t-1}\) and use \(\mu_t=\mu_{t-1}P_t\):
\[
\|P_t f\|_{2,\mu_{t-1}}^2 \le \int P_t(|f|^2)\,d\mu_{t-1} = \int |f|^2\,d(\mu_{t-1}P_t)=\int |f|^2\,d\mu_t=\|f\|_{2,\mu_t}^2.
\]
Hence \(\|P_t\|_{L^2(\mu_t)\to L^2(\mu_{t-1})}\le 1\). Restricting to mean-zero subspaces yields \(\sigma_t\le 1\).
\end{proof}

We now show how contractivity of scalar $L^2_0(\mu_t)$ functions translates to matrix-valued functions. To do this, we need to introduce the equivalent of the $L^2(\nu)$ norm for such functions.
Given $F \colon \Omega \to \mathbb{C}^{m \times m}$, we define  \vspace{-0.3cm}
\[
\|F\|_{2,\nu} \coloneqq  \left(\int \tr F^\ast F \,d\nu\right)^{1/2} = \left(\int \|F\|_{\mathrm F}^2 \,d\nu\right)^{1/2}.
\]
It can be checked this is indeed a norm.

\begin{restatable}{lemma}{SpectralLift}\label[lemma]{lem:scalar_to_matrix_spec}
Let $F \colon \Omega \to \Hm$ such that $\mu_t(F) = 0$ (the zero matrix). Then, for all $t\ge 1$,
\[
\|P_t F\|_{2,\mu_{t-1}}\le \sigma_t\,\|F\|_{2,\mu_t}.
\]
\end{restatable}
Notice that $\|F\|_{2,\nu}$ is defined with respect to the Frobenius norm of $F(x)$: this is why the Frobenius norm of the observables' oscillation $\Delta_{\mathrm F}$ appears in \cref{thm:spec}.

\begin{proof}
Write $F=(f_{ab})_{a,b=1}^m$ entrywise. Since $\mu_t(F)=0$, each scalar entry satisfies $\mu_t(f_{ab})=0$.
Moreover,
\[
\|F\|_{2,\mu_t}^2=\int \|F(x)\|_{\mathrm F}^2\,d\mu_t(x)=\sum_{a,b=1}^m \int |f_{ab}(x)|^2\,d\mu_t(x)
=\sum_{a,b=1}^m \|f_{ab}\|_{2,\mu_t}^2,
\]
and similarly for $\|P_tF\|_{2,\mu_{t-1}}^2$.
Applying the scalar contraction \eqref{eq:L2_contract_scalar} to each entry gives
\[
\|P_t f_{ab}\|_{2,\mu_{t-1}}\le \sigma_t\,\|f_{ab}\|_{2,\mu_t},
\qquad a,b=1,\dots,m.
\]
Squaring and summing over $(a,b)$ yields
\[
\|P_tF\|_{2,\mu_{t-1}}^2=\sum_{a,b}\|P_t f_{ab}\|_{2,\mu_{t-1}}^2
\le \sigma_t^2\sum_{a,b}\|f_{ab}\|_{2,\mu_t}^2
=\sigma_t^2\|F\|_{2,\mu_t}^2
\]
Taking square roots concludes the proof.
\end{proof}

We then obtain some estimates on  \(W_t\).
\begin{lemma}\label[lemma]{lem:W_bounds_spec}
Under Assumption~\ref{ass:Delta}, for all \(t,x,\phi\),
\[
\|W_t(x)\|_{\op}\le e^{\frac12 s\Delta_{\op}},
\qquad
\sup_{x,y \in \Omega} \|W_t(x)-W_t(y)\|_{\op} \le \tfrac12 s\Delta_{\op}\, e^{\frac12 s\Delta_{\op}},
\qquad
\|H_{n,1}\|_{\infty}\le s\Delta_{\op}\, e^{s\Delta_{\op}}.
\]
\end{lemma}

\begin{proof}
Let \(C_t:=\mu_t(F_t)=\E[F_t(X_t)]\). Since \(C_t\) is a convex combination of \(\{F_t(y):y\in\Omega\}\),
for every \(x\in\Omega\),
\[
\|\widetilde F_t(x)\|_{\op}=\|F_t(x)-C_t\|_{\op}
\le \int \|F_t(x)-F_t(y)\|_{\op}\,\mu_t(dy)\le \Delta_{\op}.
\]
Hence
\(
\|W_t(x)\|_{\op}=\big\|e^{\frac12 e^{i\phi}s\widetilde F_t(x)}\big\|_{\op}\le e^{\frac12 s\Delta_{\op}}.
\)

Next, since \(\widetilde F_t\) differs from \(F_t\) by a constant matrix, it has the same oscillation:
\(\sup_{x,y}\|\widetilde F_t(x)-\widetilde F_t(y)\|_{\op}\le \Delta_{\op}\).
By Lemma~\ref{lem:exp_bounds}, for all \(x,y\in\Omega\),
\[
\|W_t(x)-W_t(y)\|_{\op}
\le e^{\|\frac12 e^{i\phi}s\widetilde F_t\|_{\infty}}\cdot \Big|\tfrac12 e^{i\phi}s\Big| \cdot \|\widetilde F_t(x)-\widetilde F_t(y)\|_{\op}
\le \tfrac12 s\Delta_{\op}\, e^{\frac12 s\Delta_{\op}}.
\]
Taking \(y\sim \mu_t\) and using Jensen gives \(\|W_t(x)-\E W_t(X_t)\|_{\op}\le \tfrac12 s\Delta_{\op} e^{\frac12 s\Delta_{\op}}\).

Finally, similarly for \(\Theta_1(x)=e^{\gamma s\widetilde F_n(x)}\) with \(\gamma=\cos\phi\),
\[
\|\Theta_1(x)-\Theta_1(y)\|_{\op}
\le e^{\|\gamma s \widetilde F_n\|_{\infty}}\cdot |\gamma|s\,\|\widetilde F_n(x)-\widetilde F_n(y)\|_{\op}
\le s\Delta_{\op}\,e^{s\Delta_{\op}}.
\]
The last inequality in the statement follows by the definition \(H_{n,1}=\Theta_1-\mu_n(\Theta_1)\).
\end{proof}

We can now bound the coefficients $\{b_{n,i}\}$.

\begin{lemma}\label[lemma]{lem:coeff_spec}
Let \(\phi \in [-\pi/2,\pi/2]\) and \(0 \le s \le 1/(8\Delta_{\op})\). Then,
\begin{enumerate}
\item[(i)] \(b_{n,1}(\phi)\le \exp(2s^2\Delta_{\op}\Delta_{\mathrm F})\).
\item[(ii)] For each \(i\ge 2\),
\begin{equation}\label{eq:beta_i_spec}
b_{n,i}(\phi)
\le
4\, s^2\,\Delta_{\op}\Delta_{\mathrm F}\, e^{s\Delta_{\op}} \prod_{\ell=n-i+2}^{n} (e^{s\Delta_{\op}}\sigma_{\ell}).
\end{equation}
\end{enumerate}
\end{lemma}

\begin{proof}
The first statement follows from \cref{lem:matrix_hoeffding} and $\|\widetilde F_n\|_{\infty} \le \Delta_{\op} \le  \Delta_{\mathrm F}$.

Fix \(i\ge 2\) and \(\tau=n-i+1\). Let \(G:=P_{\tau+1}H_{n,i-1}\) so that \(\mu_\tau(G)=0\).
Then, by \cref{eq:Bi_prefactor_inhomW}
\[
\|B_{n,i}\|_{\op}
\le
4\,\norminf{W_\tau}\,\norminf{W_\tau-\mu_\tau W_\tau}\,\mu_\tau(\|G\|_{\op}).
\]
By Lemma~\ref{lem:W_bounds_spec}, \(\norminf{W_\tau}\le e^{\frac12 s\Delta_{\op}}\) and \(\norminf{W_\tau-\mu_\tau W_\tau}\le \tfrac12 s\Delta_{\op}e^{\frac12 s\Delta_{\op}}\), hence
\begin{equation}\label{eq:Bi_prefactor_spec}
\|B_{n,i}\|_{\op}\le 2 s\Delta_{\op} \,e^{s\Delta_{\op}}  \mu_\tau(\|G\|_{\op}).
\end{equation}

To bound \(\mu_\tau(\|G\|_{\op})\), use \(\|A\|_{\op}\le \|A\|_{\mathrm{F}}\), Jensen's inequality, and the linearity of the trace:
\[
\mu_\tau(\|G\|_{\op})\le \mu_\tau(\|G\|_{\mathrm{F}})\le \mu_\tau(\|G\|_{\mathrm{F}}^2)^{1/2} = \normtwo{G}{\mu_\tau}.
\]

By \cref{lem:scalar_to_matrix_spec},
\[
\normtwo{G}{\mu_\tau} = \normtwo{P_{\tau+1}H_{n,i-1}}{\mu_\tau} \le \sigma_{\tau+1}\normtwo{H_{n,i-1}}{\mu_{\tau+1}}.
\]

We now bound $\normtwo{H_{n,i}}{\mu_{\tau}}$ recursively as follows:
 \begin{equation}
 \label{eq:rec_Hbound}
\normtwo{H_{n,i}}{\mu_{\tau}}
\le e^{s\Delta_{\op}}\,\sigma_{\tau+1}\,\normtwo{H_{n,i-1}}{\mu_{\tau+1}}.
\end{equation}
To prove \cref{eq:rec_Hbound}, notice that, by definition, \(H_{n,i}=\Theta_i-B_{n,i}\) with \(B_{n,i}=\E[\Theta_i(X_\tau)]\).
Thus \(\normtwo{H_{n,i}}{\mu_\tau}\le \normtwo{\Theta_i}{\mu_\tau}\).
Using \(\Theta_i(x)=W_\tau(x)(P_{\tau+1}H_{n,i-1})(x)W_\tau(x)^\ast\) and \cref{lem:W_bounds_spec},
\[
\normtwo{\Theta_i}{\mu_\tau} \le \|W_\tau\|_{\infty}^2 \normtwo{P_{\tau+1}H_{n,i-1}}{\mu_\tau} \le e^{s\Delta_{\op}}\,\sigma_{\tau + 1}\,\normtwo{H_{n,i-1}}{\mu_{\tau+1}}.
\]

Iteratively applying \cref{eq:rec_Hbound} $i-2$ times yields
\[
\normtwo{G}{\mu_\tau}\le (e^{s\Delta_{\op}})^{i-2}\,\left(\prod_{\ell=\tau+1}^{n}\sigma_{\ell}\right)\,\normtwo{H_{n,1}}{\mu_{n}}.
\]
We next bound $\|H_{n,1}\|_{2,\mu_n}$ using Frobenius oscillation.
Let $X,X'\stackrel{iid}{\sim}\mu_n$. Since $\|\cdot\|_{2,\mu_n}$ is the $L^2(\mu_n)$ norm induced by $\|\cdot\|_{\mathrm F}$,
\[
\|H_{n,1}\|_{2,\mu_n}^2=\E\|\Theta_1(X)-\E\Theta_1(X)\|_{\mathrm F}^2=\tfrac12\,\E\|\Theta_1(X)-\Theta_1(X')\|_{\mathrm F}^2
\le \tfrac12\Big(\sup_{x,y}\|\Theta_1(x)-\Theta_1(y)\|_{\mathrm F}\Big)^2.
\]
By Lemma~\ref{lem:exp_bounds} and Assumption~\ref{ass:Delta},
\[
\sup_{x,y}\|\Theta_1(x)-\Theta_1(y)\|_{\mathrm F}
\le e^{\|\gamma s\widetilde F_n\|_{\infty}}\cdot |\gamma|s\cdot \sup_{x,y}\|\widetilde F_n(x)-\widetilde F_n(y)\|_{\mathrm F}
\le s e^{s\Delta_{\op}}\Delta_{\mathrm F}.
\]
Therefore,
\[
\|H_{n,1}\|_{2,\mu_n}\le \tfrac{s}{\sqrt2}e^{s\Delta_{\op}}\Delta_{\mathrm F}.
\]

Combining this with \eqref{eq:Bi_prefactor_spec} and our assumption $e^{2s\Delta_{\op}} \le \sqrt{2}$ yields
\begin{align*}
\|B_{n,i}\|_{\op}\le 2 s\Delta_{\op} \,e^{s\Delta_{\op}}  \mu_\tau(\|G\|_{\op})
&\le 2\sqrt2\, s^2\Delta_{\op}\Delta_{\mathrm F}\, e^{s\Delta_{\op}} e^{s\Delta_{\op}}  (e^{s\Delta_{\op}})^{i-2}\,\prod_{\ell=\tau+1}^{n}\sigma_{\ell} \\
&\le 4 s^2\Delta_{\op}\Delta_{\mathrm F} \, (e^{s\Delta_{\op}})^i \prod_{\ell=\tau+1}^{n}\sigma_{\ell}.
\end{align*}
Taking $(e^{s\Delta_{\op}})^{i-2}$ inside the product yields \cref{eq:beta_i_spec}.
\end{proof}

Equipped with the previous estimates, \cref{thm:spec} can be proved in a similar way as \cref{thm:curv}.

\ChernoffSpec*

\begin{proof}
If \(\Delta_{\op}\Delta_{\mathrm F}=0\), then \(\Delta_{\op}=0\), since \(\|\cdot\|_{\op}\le \|\cdot\|_{\mathrm F}\). Hence each \(F_j\) is constant on \(\Omega\), so every centered summand vanishes and the result is trivial. We therefore assume \(\Delta_{\op}\Delta_{\mathrm F}>0\).

Set
\[
r_{\max}:=\frac{\pi\lambda}{64\Delta_{\mathrm F}},
\qquad
C:=24\Delta_{\op}\Delta_{\mathrm F}\,\lambda^{-1}.
\]
Fix \(r\in[0,r_{\max}]\), \(\phi \in [-\pi/2,\pi/2]\), and let \(s:=4r/\pi\). Since \(\Delta_{\op}\le \Delta_{\mathrm F}\), we have \(s\Delta_{\op}\le \lambda/16\) and \(s\le 1/(8\Delta_{\op})\). Thus \cref{lem:coeff_spec} applies, and
\begin{align*}
\sum_{i=1}^n b_{n,i}(\phi)
&\le e^{2s^2\Delta_{\op}\Delta_{\mathrm F}}
   +4s^2\Delta_{\op}\Delta_{\mathrm F}e^{2s\Delta_{\op}}
     \sum_{i=2}^n \prod_{\ell=n-i+2}^{n}(e^{s\Delta_{\op}}\sigma_\ell)  \\
&\le 1+8s^2\Delta_{\op}\Delta_{\mathrm F}
     \sum_{i=1}^n \prod_{\ell=n-i+2}^{n}(e^{\lambda/8}\sigma_\ell) \\
&\le 1+C s^2.
\end{align*}
In the last step we apply \cref{lem:weighted_renewal} with \(\kappa_\ell=1-\sigma_\ell\) and \(\kappa=\lambda\), using \cref{ass:lambda}; the tilt \(e^{\lambda/8}\) is bounded by the allowed tilt \(e^{\pi\lambda/24}\).

By \cref{lem:close},
\[
\E_{\mu_0}\tr\exp\!\Bigl(r\sum_{j=1}^n \widetilde F_j(X_j)\Bigr)
\le m^{2-\pi/4}\exp\!\Bigl(\frac{16C}{\pi^2}n r^2\Bigr)
= m^{2-\pi/4}\exp\!\Bigl(\frac12 v_B^2 n r^2\Bigr),
\]
where \(v_B^2=\frac{768}{\pi^2}\frac{\Delta_{\op}\Delta_{\mathrm F}}{\lambda}\). Hence, for every \(r\in[0,r_{\max}]\),
\[
\Pbbmu\!\Bigl(\lmax\Bigl(\sum_{j=1}^n \widetilde F_j(X_j)\Bigr)\ge n\varepsilon\Bigr)
\le
m^{2-\pi/4}\exp\!\Bigl(-n\bigl(r\varepsilon-\tfrac12 v_B^2r^2\bigr)\Bigr).
\]

If \(0\le \varepsilon\le \Delta_{\op}\), then
\[
r^\star:=\frac{\varepsilon}{v_B^2}
\le \frac{\Delta_{\op}}{v_B^2}
=\frac{\pi^2\lambda}{768\Delta_{\mathrm F}}
\le \frac{\pi\lambda}{64\Delta_{\mathrm F}}
=r_{\max},
\]
because \(\pi\le 12\). Choosing \(r=r^\star\) gives
\[
\Pbbmu\!\Bigl(\lmax\Bigl(\sum_{j=1}^n \widetilde F_j(X_j)\Bigr)\ge n\varepsilon\Bigr)
\le m^{2-\pi/4}\exp\left(-\frac{n\varepsilon^2}{2v_B^2}\right).
\]

Finally, for every \(j\),
\[
\|\widetilde F_j(x)\|_{\op}
=\|F_j(x)-\mu_j(F_j)\|_{\op}
\le \int \|F_j(x)-F_j(y)\|_{\op}\,\mu_j(dy)
\le \Delta_{\op}.
\]
Therefore
\[
\lmax\Bigl(\sum_{j=1}^n \widetilde F_j(X_j)\Bigr)
\le \sum_{j=1}^n \|\widetilde F_j(X_j)\|_{\op}
\le n\Delta_{\op},
\]
so the bound is trivial when \(\varepsilon>\Delta_{\op}\).
\end{proof}

%% file: elo.tex
\section{An analysis of the Elo rating system with evolving skills}
\label{sec:elo}

The Elo rating system is a popular method for estimating the relative skill of
players in sports analytics, particularly chess and tennis. It was introduced by
Arpad Elo in the 1950s to rank chess players by assigning them a numerical
score, which is updated according to a simple formula: after a game between two
players \(i\) and \(j\), if the former wins, their rating is increased by an
amount proportional to the probability estimated by the model that \(i\) would
have lost to \(j\); the rating of \(j\) is instead decreased by the same amount.
Unexpected outcomes result in larger changes in ratings.

The effectiveness, simplicity, and interpretability of Elo ratings have made
them, arguably, the most popular ranking system in sports analytics.  More
recently, Elo ratings have also been applied to evaluate and rank large language
models~\citep{boubdir2024elo}.  Despite their popularity, however, not much
attention has been devoted to studying Elo ratings from a theoretical point of
view.  Recently, \citet{EloOurs} have studied Elo ratings under the
Bradley-Terry-Luce (BTL) model, a popular statistical model of match outcomes.
Using techniques from Markov chain theory, they show Elo ratings
well-approximate the true ratings of the players with a sample complexity that
is competitive against the state of the art of BTL estimation.

While these results offer a compelling explanation of the effectiveness of Elo
ratings, they suffer from an important limitation: they require the true skills
of players to remain fixed over time.  Therefore, they fail to investigate one
of the most important properties of an online ranking system: the capability to
detect and track changes in the skills of the players.  To address this gap, we
consider an extension of the BTL model in which the players' true ratings change
over time.  In the setup below, the true ratings and matchup distributions
evolve as an external time-inhomogeneous Markov chain, independently of the
randomness used to select pairs and generate match outcomes.  We give a
sufficient condition under which the joint process containing the Elo
ratings, true ratings, and matchup distributions is itself positively curved,
and then apply the concentration results to show that Elo ratings track the
changing true ratings over time.

\subsection{Related work}
 There is a plethora of work on the Bradley-Terry-Luce model and the problem of learning from pairwise comparison in general. Here we discuss only the most relevant literature. \citet{A:elo-rats,A:elo-unpublished} study Elo rating systems (in the homogeneous setting)  from an applied probability perspective; in particular, the existence and convergence to a stationary distribution of the Elo process is established. \citet{EloOurs} prove Elo recovers the (static) BTL parameters with a sample complexity close to the state-of-the-art. This is done by studying Elo as a Markov chain and bounding the rate of convergence of this (time-homogeneous) chain in a conceptually similar way as  the present work. Since Elo does not converge in total variation, this is done by relating the Elo process to an auxiliary \emph{noisy} version of Elo, so that Chernoff bounds for Markov chains~\citep{L:chernoff-markov-disc,P:conc-psgap} can be applied. This is because the results by \citet{JoulinOllivier} have a suboptimal dependency on the curvature of the chain. Our concentration results, besides working in the more general inhomogeneous setting, have an optimal dependency on the curvature, allowing us to study Elo directly.

Compared to the static setup, theoretical results for BTL estimation in the dynamic setting, where model parameters change over time,  are limited.  \citet{DynamicBTL} adapt the classical Rank Centrality algorithm  for static BTL~\citep{RankCentrality} to the dynamic setting essentially by maintaining the average win-rates between any two players over a sliding window. While it is difficult to present a like-for-like comparison with our results, we note that their algorithm has the following drawbacks compared to ours: it does not maintain an online estimator of the ratings; it requires a very strong Lipschitz condition on how fast the ratings change that needs to be satisfied uniformly for \emph{each pair} of players.

\citet{pmlr-v108-bong20a} develop an algorithm for a similar setup as ours that works by performing kernel smoothing to pre-process the pairwise comparisons over time and then solve a maximum log-likelihood estimation problem for each time step. This makes estimation much more computationally expensive than the intrinsically online Elo rating system.

\subsection{Elo ratings and the dynamic BTL model}

We denote the number of players by \(n\ge2\).  Let \(M\ge0\).  For any time
\(t\in\mathbb N\), we assume each player \(i\) has a \emph{true} rating
\(\rho_i^t\), and denote by \(\rho^t\) the vector of true ratings at time \(t\).
The goal of Elo is to approximate these ratings.  At each time \(t\), a pair of
players is selected according to a probability distribution \(q^t\) over the
\(\binom n2\) unordered pairs of players.
The sequence \((\rho^t,q^t)_{t\in\mathbb N}\) can evolve as a time-inhomogeneous Markov chain, as
long as their evolution is independent of the randomness of matchup selection
and match outcomes.

We assume there exists \(M>1\) such that \(\rho_i^t\in[-M,M]\) for all players
\(i\) and steps \(t\).  This is reasonable since it just means that ratings
cannot diverge.  As typical in the literature, we also assume \(M\) is known.
We further assume that, for all \(t\), \(\sum_{i=1}^n\rho_i^t=0\).  This is
without loss of generality and is needed to be able to approximate the true
rating since it is impossible to distinguish between a vector of true ratings
\(\rho\) and the shifted vector \(\rho+c\vec 1\).

We will denote by \(X_i^t\) the Elo rating of player \(i\) after the first
\(t\) matches have been played.  For simplicity, we initialise \(X_i^0=0\).

We can now formally define the Elo process according to our version of the
dynamic Bradley Terry Luce model.  For any \(t\ge1\), \vskip 0.2cm

1. Select a pair of players \(\{I,J\}\) at random according to \(q^t\):
\(\Pbb(\{I,J\}=\{i,j\}\mid q^t)=q^t_{\{i,j\}}\). \vskip 0.2cm

2. The outcome of the match between \(I\) and \(J\) is a Bernoulli random
variable, conditionally independent of past outcomes given the current
environment, where the probability that \(I\) beats \(J\) is
\(\sigma(\rho_I^t-\rho_J^t)=1/(1+\exp(\rho_J^t-\rho_I^t))\). \vskip 0.2cm

3. Suppose \(I\) beats \(J\).  We first update the Elo ratings of \(I\) and
\(J\) as follows (if instead \(J\) beats \(I\), we swap the roles of \(I\) and
\(J\)):
\begin{align*}
\widehat X_I^t
&\leftarrow X_I^{t-1}+\eta\sigma(X_J^{t-1}-X_I^{t-1}), \quad \text{and}\\
\widehat X_J^t
&\leftarrow X_J^{t-1}-\eta\sigma(X_J^{t-1}-X_I^{t-1}),
\end{align*}
where \(\eta\in(0,1/2)\) is a step-size chosen in advance. All other ratings
remain unchanged: \\
\(
\widehat X_k^t\leftarrow X_k^{t-1} \text{ for all } k\not\in\{I,J\}.
\) \vskip 0.2cm

4. Orthogonally project the vector of ratings \(\widehat X^t\) to
\(\mathcal X_M \coloneqq [-M,M]^n\cap\{x\in\mathbb R^n\colon x \perp \mathbf 1\}\), and call the resulting
vector \(X^t\).\footnote{This is needed to guarantee that Elo ratings do not
diverge and the chain remains in a set with good curvature. We refer to
\citep{EloOurs} for further discussions about this projection step.} \vskip 0.2cm

We next formalise the evolution of the environment \((\rho^t, q^t)_{t\in\mathbb N}\).

\subsection{The Markov chain model and curvature assumptions}

Let \(\mathcal Q\) be a closed
subset of the simplex of probability distributions on unordered pairs \(\{i,j\}\),
\(1\le i<j\le n\), which represents the set of possible matchup distributions.  We denote a state of the environment as
\(
e=(\rho,q)\in \mathcal E\coloneqq \mathcal X_M\times \mathcal Q.
\)
We equip \(\mathcal E\) with the metric
\[
d_{\mathcal E}\bigl((\rho,q),(\widetilde \rho,\widetilde q)\bigr)
\coloneqq
\|\rho-\widetilde \rho\|_2
+2\sqrt 2\,\|q-\widetilde q\|_{\mathrm{TV}},
\]
where
\[
\|q-\widetilde q\|_{\mathrm{TV}}
\coloneqq
\frac12\sum_{1\le i<j\le n}
\left|q_{\{i,j\}}-\widetilde q_{\{i,j\}}\right|.
\]
We denote the joint state of the Elo ratings together with an environment as \(z=(x,e)\in\mathcal Z\coloneqq \mathcal X_M\times
\mathcal E\). We equip $\mathcal Z$ with the metric
\[
d_{\mathcal Z}\bigl((x,e),(\widetilde x,\widetilde e)\bigr)
\coloneqq
\|x-\widetilde x\|_2+d_{\mathcal E}(e,\widetilde e).
\]
We write \(D_{\mathcal Z}\) for the diameter of \((\mathcal Z,d_{\mathcal Z})\).
In particular, \(D_{\mathcal Z}\le 4M\sqrt n+2\sqrt 2\).

The environment at time \(t\) is the pair
\(
E^t\coloneqq(\rho^t,q^t).
\)
We assume that \((E^t)_{t\ge1}\) is a time-inhomogeneous Markov chain on
\(\mathcal E\), with transition kernels \(\Gamma_t\), so that
\[
\Pbb(E^{t+1}\in A\mid E^1,\ldots,E^t)=\Gamma_t(E^t,A).
\]
We formalise the independence between the environment and the Elo randomness through the following
filtrations.  Let
\(
\mathcal G_t\coloneqq \sigma(E^1,\ldots,E^t)
\)
be the information generated by the environment up to time \(t\).  In the
formal construction, after a pair \(A_s\) is selected, we write
\(
A_s=\{I_s,J_s\}
\text{ with } I_s<J_s,
\)
and let \(Y_s=1\) denote the event that player \(I_s\) beats player \(J_s\).
Let \(\mathcal F_t\) be the information available just before the match played
at time \(t\):
\[
\mathcal F_t\coloneqq
\sigma\bigl(E^1,\ldots,E^t,X^0,\{A_s,Y_s:1\le s<t\}\bigr).
\]
Notice \(X^{t-1}\) is
\(\mathcal F_t\)-measurable.  The
conditional probabilities in the numbered description above are to be
understood as
\[
\Pbb(A_t=\{i,j\}\mid \mathcal F_t)=q^t_{\{i,j\}},
\]
and, after writing the selected pair as \(A_t=\{I_t,J_t\}\) with \(I_t<J_t\),
\[
\Pbb(Y_t=1\mid \mathcal F_t,A_t=\{I_t,J_t\})
=\sigma(\rho_{I_t}^t-\rho_{J_t}^t).
\]
Finally, \(E^{t+1}\) is sampled from \(\Gamma_t(E^t,\cdot)\), conditionally
independently of \((A_t,Y_t)\) given \(\mathcal F_t\).  Equivalently, the
joint Elo and environment process
\(
Z^t\coloneqq (X^t,E^t)
\)
is a time-inhomogeneous Markov chain on \(\mathcal Z\).

For a fixed environment \(e=(\rho,q)\), let \(K_e(x,\cdot)\) denote the law
of an Elo update from current Elo vector \(x\).
Therefore, the joint transition kernel from \(Z^t\) to \(Z^{t+1}\) is
\[
P_t\bigl((x,e),d x' d e'\bigr)
=\Gamma_t(e,d e')\,K_{e'}(x,d x').
\]

\begin{remark}
In the construction above, the environment process \((\rho^t,q^t)_{t\ge1}\)
evolves independently of the Elo ratings.  This is for simplicity and can potentially be relaxed.
For example, suppose the matchup
distribution is chosen as a function of the current Elo ratings and true
ratings, say
\(
q^t=Q_t(X^{t-1},\rho^t).
\)
 Our techniques
can be adapted as long as perturbing the current Elo vector cannot cause a
large perturbation of the matchup distribution, i.e., if the maps \(Q_t\) are sufficiently Lipschitz.
\end{remark}

Our results will rely on several assumptions on the curvature of the process. To state these assumptions formally, we need to introduce the following
quantities.  For an environment \(e=(\rho,q)\), define
\[
\lambda(e)\coloneqq
\min_{\substack{v\in\R^n\setminus\{\underline{0}\}\\v\perp \mathbf 1}}
\frac{\sum_{\{i,j\}}q_{\{i,j\}}(v_i-v_j)^2}{\sum_i v_i^2}.
\]
Equivalently, \(\lambda(e)\) is the second smallest eigenvalue of the Laplacian
of the continuous-time random walk on \([n]\) with transition rates
\((q_{\{i,j\}})\).  This essentially quantifies how fast information about results between two
players propagates to the rest of the players, and is a standard quantity in
BTL estimation~\citep{pmlr-v38-shah15,RankCentrality,LSR:btl-mle}.
Notice it holds $\lambda(e) = O(1/n)$ since $\sum_{\{i,j\}} q_{\{i,j\}} = 1$,

We assume there is a deterministic constant \(\lambda>0\)
such that
\[
\lambda(e)\ge \lambda
\qquad\text{for every }e\in\mathcal E.
\]

The following quantity is a lower bound on the curvature of Elo when the
environment is held fixed, that is, it quantifies the rate of convergence in
Wasserstein distance after one step of the Elo update:
\[
\kappa\coloneqq \frac{1}{8}\,\eta e^{-4M}\lambda.
\]

We assume that the environment also contracts in Wasserstein distance.  More
precisely, we assume the environment chain has one-step
curvature at least \(\nu\in(0,1]\) with respect to \(d_{\mathcal E}\):
\[
W_1^{d_{\mathcal E}}\bigl(\Gamma_t(e,\cdot),\Gamma_t(\widetilde e,\cdot)\bigr)
\le
(1-\nu)d_{\mathcal E}(e,\widetilde e)
\qquad
\text{for all }t\ge1\text{ and }e,\widetilde e\in\mathcal E.
\]
Here \(W_1^{d_{\mathcal E}}\) denotes Wasserstein distance computed using the
metric \(d_{\mathcal E}\).  Informally, positive curvature means that the
environment dynamics contract differences between possible realisations of the
true ratings and matchup distributions over time, which we believe is a reasonable assumption.
Finally, we choose the Elo step size so
that
\[
\eta\le \frac{\nu}{2}.
\]
We will show below that the same parameter \(\kappa\) controls both how quickly
Elo forgets its initial ratings and how strongly the joint process
\((X^t,\rho^t,q^t)\) concentrates around its mean.

We start by recalling the curvature of the Elo chain when the environment is fixed.

\begin{lemma}[\citet{EloOurs}]
\label[lemma]{lem:elo_fixed_environment_curvature}
For every \(t\ge1\), every environment \(e \in \mathcal E\), and
every \(x,\widetilde x\in\mathcal X_M\),
\[
W_1^{\|\cdot\|_2}
\bigl(K_e(x,\cdot),K_e(\widetilde x,\cdot)\bigr)
\le
(1-\kappa)\|x-\widetilde x\|_2.
\]
\end{lemma}

We then bound how sensitive the Elo update is to changes in the environment.

\begin{lemma}
\label[lemma]{lem:elo_environment_sensitivity}
For every \(t\ge1\), every \(x\in\mathcal X_M\), and every
\(e,\widetilde e\in\mathcal E\),
\[
W_1^{\|\cdot\|_2}\bigl(K_e(x,\cdot),K_{\widetilde e}(x,\cdot)\bigr)
\le
\eta\,d_{\mathcal E}(e,\widetilde e).
\]
\end{lemma}

\begin{proof}
Write \(e=(\rho,q)\) and \(\widetilde e=(\widetilde \rho,\widetilde q)\).
First suppose \(q=\widetilde q\).  Couple the selected pair of players to be the same and
use the same uniform random variable to generate the two match outcomes.  For a
fixed pair \(\{i,j\}\), the probability that the two outcomes disagree is at
most
\[
\left|\sigma(\rho_i-\rho_j)-\sigma(\widetilde\rho_i-\widetilde\rho_j)\right|
\le
\frac14\left|(\rho_i-\rho_j)-(\widetilde\rho_i-\widetilde\rho_j)\right|
\le
\frac{1}{2\sqrt2}\|\rho-\widetilde\rho\|_2.
\]
If the outcomes agree, the two updates from the same \(x\) coincide.  If they
disagree, the two unprojected updates differ by at most \(\sqrt2\,\eta\) in
Euclidean norm, and the projection onto \(\mathcal X_M\) is non-expansive.
Hence this part contributes at most
\(
\frac{\eta}{2}\|\rho-\widetilde\rho\|_2.
\)

Now suppose \(\rho=\widetilde\rho\).  Couple the selected pairs maximally, so
that the probability of selecting different pairs is
\(\|q-\widetilde q\|_{\mathrm{TV}}\).  If the selected pair is the same, use the
same outcome.  If the selected pairs differ, each one-step Elo update
lies within distance \(\sqrt2\,\eta\) of \(x\), and therefore the two updated
vectors are within distance \(2\sqrt2\,\eta\) of each other.  This contributes
at most
\(
2\sqrt2\,\eta\|q-\widetilde q\|_{\mathrm{TV}}.
\)
For general \(e=(\rho,q)\) and
\(\widetilde e=(\widetilde\rho,\widetilde q)\), introduce the intermediate
environment
\(
e^\star\coloneqq(\widetilde\rho,q).
\)
By the triangle inequality for Wasserstein distance,
\begin{align*}
W_1^{\|\cdot\|_2}
\bigl(K_{(\rho,q)}(x,\cdot),
K_{(\widetilde\rho,\widetilde q)}(x,\cdot)\bigr)
&\le
W_1^{\|\cdot\|_2}
\bigl(K_{(\rho,q)}(x,\cdot),
K_{(\widetilde\rho,q)}(x,\cdot)\bigr)
+W_1^{\|\cdot\|_2}
\bigl(K_{(\widetilde\rho,q)}(x,\cdot),
K_{(\widetilde\rho,\widetilde q)}(x,\cdot)\bigr) \\
&\le
\frac{\eta}{2}\|\rho-\widetilde\rho\|_2
+2\sqrt2\,\eta\|q-\widetilde q\|_{\mathrm{TV}} \\
&\le \eta d_{\mathcal E}\bigl((\rho,q),(\widetilde\rho,\widetilde q)\bigr).
\end{align*}
\end{proof}

We can now establish a bound on the curvature of the joint Elo and environment chain.

\begin{lemma}
\label[lemma]{lem:elo_augmented_curvature}
For every \(t\ge1\) and every \(z,\widetilde z\in\mathcal Z\),
\[
W_1^{d_{\mathcal Z}}\bigl(P_t(z,\cdot),P_t(\widetilde z,\cdot)\bigr)
\le
(1-\kappa)d_{\mathcal Z}(z,\widetilde z).
\]
\end{lemma}

\begin{proof}
Write \(z=(x,e)\) and \(\widetilde z=(\widetilde x,\widetilde e)\).  Choose an optimal coupling of
\(E'\sim\Gamma_t(e,\cdot)\) and
\(\widetilde E'\sim\Gamma_t(\widetilde e,\cdot)\).  Then
\[
\E d_{\mathcal E}(E',\widetilde E')
=
W_1^{d_{\mathcal E}}\bigl(\Gamma_t(e,\cdot),
\Gamma_t(\widetilde e,\cdot)\bigr)
\le
(1-\nu)d_{\mathcal E}(e,\widetilde e).
\]

Conditional on \(E'\) and \(\widetilde E'\), choose a coupling of
$X'\sim K_{E'}(x,\cdot)$ and $\widetilde X'\sim K_{\widetilde E'}(\widetilde x,\cdot)$
\[
X'\sim K_{E'}(x,\cdot),
\qquad
\widetilde X'\sim K_{\widetilde E'}(\widetilde x,\cdot).
\]
By combining \cref{lem:elo_fixed_environment_curvature} with
\cref{lem:elo_environment_sensitivity}, we may choose this conditional
coupling so that
\[
\E\!\left[
\|X'-\widetilde X'\|_2
\middle| E',\widetilde E'
\right]
\le
(1-\kappa)\|x-\widetilde x\|_2
+\eta d_{\mathcal E}(E',\widetilde E').
\]

Summing the contribution from both the environment and the Elo ratings in the metric \(d_{\mathcal Z}\), we obtain
\[
W_1^{d_{\mathcal Z}}\bigl(P_t(z,\cdot),P_t(\widetilde z,\cdot)\bigr)
\le
(1-\kappa)\|x-\widetilde x\|_2
+(1+\eta)(1-\nu)d_{\mathcal E}(e,\widetilde e).
\]
Since \(\eta\le\nu/2\) and \(\kappa\le\eta\),
\[
(1+\eta)(1-\nu)\le 1-\nu+\eta\le 1-\frac{\nu}{2}\le 1-\kappa,
\]
and the claim follows.
\end{proof}

\subsection{Main results}

To ensure Elo is able to track the evolving true ratings, we need to quantify
how fast the true ratings can change.  To this aim, we assume there exists a
deterministic constant \(\Delta\ge0\) such that, almost surely, for every
\(t\ge1\),
\[
\E\!\left[
\|\rho^{t+1}-\rho^t\|_2^2
+4M\|\rho^{t+1}-\rho^t\|_1
\middle| \mathcal F_t
\right]
\le \Delta.
\]
This quantity measures how much the moving target can change in one step.   Unsurprisingly, a smaller value of the drift will allow Elo to
track the true ratings more closely.
\footnote{ The
\(\ell_1\) term comes from the cross term
\(\langle X^t-\rho^t,\rho^t-\rho^{t+1}\rangle\), which can be upper bounded by
\(2M\|\rho^{t+1}-\rho^t\|_1\) because both \(X^t\) and \(\rho^t\) belong to
\([-M,M]^n\). We could potentially upper bound the cross term  by:
\(\langle X^t-\rho^t,\rho^t-\rho^{t+1}\rangle
\le \frac{\varepsilon}{2}\|X^t-\rho^t\|_2^2
+\frac{1}{2\varepsilon}\|\rho^t-\rho^{t+1}\|_2^2\), for all
\(\varepsilon>0\).  Choosing \(\varepsilon=\kappa\) would allow us to require
a bound only on the simpler quantity
\(\E[\|\rho^t-\rho^{t+1}\|_2^2\mid\mathcal F_t]\), but it would also increase
the dependency on \(\kappa^{-1}\) from linear to quadratic.}

The first lemma shows Elo rating track the true skills in expectation.  It exploits contraction
of the Elo update with a fixed environment.

\begin{lemma}
\label[lemma]{lem:elo_augmented_expect}
For every \(t\ge1\), it holds
\[
\E\|X^t-\rho^t\|_2^2
\le
(1-\kappa)^{t-1}\|X^0-\rho^1\|_2^2
+\frac{\Delta}{\kappa}
+\frac{2\eta^2}{\kappa}.
\]
\end{lemma}

\begin{proof}
\citet[Proof of Theorem 2.7]{EloOurs} have shown that, for a fixed environment,  an Elo update satisfies
\[
\E\!\left[
\|X^t-\rho^t\|_2^2
\middle| \mathcal F_t
\right]
\le
(1-\kappa)\|X^{t-1}-\rho^t\|_2^2+2\eta^2.
\]
For \(t\ge2\), since \(X^{t-1}\) and \(\rho^{t-1}\) both belong to
\(\mathcal X_M\),
\[
2\langle X^{t-1}-\rho^{t-1},\rho^{t-1}-\rho^t\rangle
\le
4M\|\rho^t-\rho^{t-1}\|_1.
\]
Therefore
\[
\|X^{t-1}-\rho^t\|_2^2
\le
\|X^{t-1}-\rho^{t-1}\|_2^2
+\|\rho^t-\rho^{t-1}\|_2^2
+4M\|\rho^t-\rho^{t-1}\|_1.
\]
Using the drift assumption and then taking expectations,
\[
\E\|X^t-\rho^t\|_2^2
\le
(1-\kappa)\E\|X^{t-1}-\rho^{t-1}\|_2^2+2\eta^2+\Delta.
\]
The case \(t=1\) satisfies the same bound without the drift term, with
\(\|X^0-\rho^1\|_2^2\) as the initial error.  Iterating the recursion and summing
the resulting geometric series proves the claim.
\end{proof}

We now apply \cref{thm:conc} to obtain concentration of the Elo rating at time t around the true rating.
\cref{thm:conc} is a pointwise estimate and require a bound on the \emph{granularity} of the chain,
which essentially implies the environment (and the true ratings in particular) can never change abruptly.
More precisely, we assume there exists
deterministic constants \(h_\rho,h_q\ge0 < \infty\) such that, for every
\(t\ge1\), every \(e=(\rho,q)\in\mathcal E\), and every
\((\rho',q')\in\operatorname{Supp}\Gamma_t(e,\cdot)\),
\[
\|\rho'-\rho\|_2\le h_\rho,
\qquad
\|q'-q\|_{\mathrm{TV}}\le h_q.
\]
Then, the one-step support diameter of the augmented chain, computed in
\(d_{\mathcal Z}\), is bounded by
\[
2\sqrt2\,\eta+2h_\rho+4\sqrt2\,h_q.
\]

\begin{theorem}
\label[theorem]{thm:elo_augmented_point}
Let \(B\coloneqq 2\sqrt2\,\eta+2h_\rho+4\sqrt2\,h_q\).  There exists a
universal constant \(C>0\) such that, for every \(\epsilon>0\), if the burn-in
time satisfies
\[
t\ge C\kappa^{-1}\log(nM\epsilon^{-1}\eta^{-1}),
\]
then
\[
\Pbb\left(
\begin{aligned}
\|X^t-\rho^t\|_2
&\ge
\sqrt{\frac{\Delta}{\kappa}}
+(1+\epsilon)\sqrt{\frac{2\eta^2}{\kappa}}
+C\epsilon\,\frac{B}{\sqrt\kappa}
\end{aligned}
\right)
\le
2e^{-\epsilon^2}.
\]
\end{theorem}

\begin{proof}
The function \(f:\mathcal Z\to\R\) defined by
\[
f(x,\rho,q)=\|x-\rho\|_2
\]
is \(1\)-Lipschitz with respect to \(d_{\mathcal Z}\).  By
\cref{lem:elo_augmented_expect} and Jensen's inequality,
\[
\E f(Z^t)
\le
(1-\kappa)^{(t-1)/2}\|X^0-\rho^1\|_2
+\sqrt{\frac{\Delta}{\kappa}}
+\sqrt{\frac{2\eta^2}{\kappa}}.
\]
By \cref{lem:elo_augmented_curvature}, the joint chain satisfies the
positive-curvature assumption with curvature \(\kappa\).  Applying
\cref{thm:conc} to the scalar observable \(f\) gives
\[
\Pbb\left(f(Z^t)\ge \E f(Z^t)+r\right)
\le
2\exp\left(-\frac{r^2\kappa}{8B^2}\right).
\]
Choosing a large enough $C>0$ and taking \(r=C\epsilon B/\sqrt\kappa\) together with the condition on the burn-in time $t$
proves the claim.
\end{proof}

\citet{EloOurs}
proved that, when the ratings do not change over time, $X^t$ is at a distance squared of $\approx \eta^2/\kappa$ from the true ratings. \cref{thm:elo_augmented_point} generalises their result with an additional correction of $(B^2+\Delta^2)/ \kappa$ when the ratings change over time.

\cref{thm:elo_augmented_point} allows us to reason about the optimal choice of (a fixed) step size. Indeed, since $\kappa$ is linearly proportional to the step size $\eta$, ignoring the granularity terms, the bound in \cref{thm:elo_augmented_point} is minimised when $\eta = \Theta(\sqrt{\Delta})$. This confirms the intuition that you want to choose the step size as small as possible to minimise the bias, but large enough so that it can keep track of the changes in the ratings.

The previous result tells us that Elo can track the ratings of the players, as long as they do not change too fast. Players' form, however, ebbs and flows during a season. Furthermore, in sports like tennis, different periods in a season are associated with different court surfaces and certain players might perform better than others on a particular surface. Therefore, in many situations, we might want to average players' Elo ratings over a longer period of time. We apply \cref{thm:curv} to show that Elo ratings, averaged after a sufficiently long \emph{burn-in}, concentrates around the averaged true ratings. This result has the additional benefit that it does not require a bound on the granularity of the chain: the true ratings can change abruptly on occasions as long as the expected drift $\Delta$ is small.

\begin{theorem}[Averaged Elo ratings with a Markovian environment]
\label[theorem]{thm:elo_augmented_avg}
There exists a universal constant \(C>0\) such that, for every
\(\epsilon>0\), \(\delta\in(0,1)\), every \(T_0\ge1\), and every
\[
T\ge
C\epsilon^{-2}\eta^{-2}M^2n\log(n/\delta),
\]
with probability at least \(1-\delta\),
\[
\left\|
\frac1T\sum_{k=T_0+1}^{T_0+T}X^k
-
\frac1T\sum_{k=T_0+1}^{T_0+T}\rho^k
\right\|_2
\le
\sqrt{\frac{\Delta}{\kappa}}
+(1+\epsilon)\sqrt{\frac{2\eta^2}{\kappa}}.
\]
\end{theorem}

\begin{proof}
By \cref{lem:elo_augmented_expect} and Jensen's inequality,
\begin{equation}
\label{eq:elo_point_bias}
\|\E(X^k-\rho^k)\|_2
\le
(1-\kappa)^{(k-1)/2}\|X^0-\rho^1\|_2
+\sqrt{\frac{\Delta}{\kappa}}
+\sqrt{\frac{2\eta^2}{\kappa}}.
\end{equation}
By Cauchy-Schwarz and the geometric sum bound
\(\sum_{k\ge1}(1-\kappa)^{k-1}\le\kappa^{-1}\),
\[
\frac1T\sum_{k=T_0+1}^{T_0+T}
(1-\kappa)^{(k-1)/2}\|X^0-\rho^1\|_2
\le
\frac{\|X^0-\rho^1\|_2}{\sqrt{\kappa T}}.
\]
Since \(\|X^0-\rho^1\|_2\le M\sqrt n\), the
lower bound on \(T\), for a large enough \(C>0\), makes this
term at most
\((\epsilon/2)\sqrt{2\eta^2/\kappa}\).  Averaging \eqref{eq:elo_point_bias} over
\(k=T_0+1,\ldots,T_0+T\) therefore gives
\[
\left\|
\E\left[\frac1T\sum_{k=T_0+1}^{T_0+T}(X^k-\rho^k)\right]
\right\|_2
\le
\sqrt{\frac{\Delta}{\kappa}}
+\left(1+\frac{\epsilon}{2}\right)\sqrt{\frac{2\eta^2}{\kappa}}.
\]
Now define \(F:\mathcal Z\to \mathsf H_{n+1}\) by
\[
F(x,\rho,q)=
\begin{pmatrix}
0 & (x-\rho)^T\\
x-\rho & \mathbf 0
\end{pmatrix}.
\]
Therefore, \(\|F(x,\rho,q)\|_{\op}=\|x-\rho\|_2\) and
\(\Lip(F)^{\op}\le1\) with respect to \(d_{\mathcal Z}\).  Applying
\cref{thm:curv} to the joint Elo and environment chain, we obtain
\[
\Pbb\left(
\left\|
\frac1T\sum_{k=T_0+1}^{T_0+T}(X^k-\rho^k)
-
\E\left[\frac1T\sum_{k=T_0+1}^{T_0+T}(X^k-\rho^k)\right]
\right\|_2
>r
\right)
\le
(n+1)^{2-\pi/4}
\exp\left(-\frac{\pi^2T\kappa r^2}{384D_{\mathcal Z}^2}\right).
\]
Since \(D_{\mathcal Z}\le 4M\sqrt n+2\sqrt2\), we have
\(D_{\mathcal Z}^2\le C_0M^2n\) for some universal constant \(C_0\).
Taking \(r=(\epsilon/2)\sqrt{2\eta^2/\kappa}\), and choosing the universal
constant \(C\) in the lower bound on \(T\) large enough, makes the RHS of the last display
at most \(\delta\).  This proves the claim.
\end{proof}

\begin{remark}
The constants \(\lambda,\nu,\Delta,h_\rho,h_q\) are taken to be uniform in time
only for simplicity of notation.  Our arguments also apply with deterministic
time-dependent bounds, at the cost of replacing the geometric factors appearing
above by the corresponding products and weighted sums.
\end{remark}

%% file: useful.tex
\section{Auxiliary results}
We prove a generalisation of Hoeffding's lemma to the Hermitian setting. While similar results are well known (see, e.g., \cite{userfriendly}), we have not found a reference for exactly the following result.

\begin{lemma}[Hermitian Hoeffding's lemma]\label[lemma]{lem:matrix_hoeffding}
Let \(Y\) be a random matrix in \(\Hm\) such that \(\E Y=0\) and assume there exists $R > 0$ such that \(-R\Id\preceq Y\preceq R\Id\) almost surely.
Then, for all \(s\in\R\),
\[
\E\big[\exp(sY)\big]\ \preceq\ \cosh(sR)\,\Id\ \preceq\ \exp\!\Big(\frac{s^2R^2}{2}\Big)\Id.
\]
In particular, \(\|\E[\exp(sY)]\|_{\op}\le \exp(s^2R^2/2)\) and
\(\E[\tr \exp(sY)]\le m\,\exp(s^2R^2/2)\).
\end{lemma}

\begin{proof}
Diagonalize \(Y=U\Lambda U^\ast\) with \(\Lambda=\mathrm{diag}(\lambda_1,\dots,\lambda_m)\) and \(\lambda_i\in[-R,R]\).
For \(y\in[-R,R]\), convexity of \(e^{sy}\) gives
\(e^{sy}\le \frac{R+y}{2R}e^{sR}+\frac{R-y}{2R}e^{-sR}\).
Applying this inequality to each eigenvalue of $Y$,
we obtain
\[
e^{sY}\ \preceq\ \frac{R\Id+Y}{2R}e^{sR} +\frac{R\Id-Y}{2R}e^{-sR} .
\]
Taking expectations and using \(\E Y=0\) yields
\(\E e^{sY}\preceq \tfrac12(e^{sR}+e^{-sR})\Id=\cosh(sR)\Id\).
Finally \(\cosh(u)\le e^{u^2/2}\) for all \(u\in\R\), so \(\cosh(sR)\Id\preceq e^{s^2R^2/2}\Id\).
\end{proof}

We often need to bound the norm of (the difference of) matrix exponentials. We will use the following well known results, of which we include a proof for completeness.

\begin{lemma}[Matrix exponential bounds]\label[lemma]{lem:exp_bounds}
For any \(H\in\mathbb{C}^{m \times m}\), \(\|e^H\|_{\op} \le e^{\|H\|_{\op}}\).
For any \(A,B\in\mathbb{C}^{m \times m}\),
\[
\|e^{A}-e^{B}\|_{\op}\le e^{\max\{\|A\|_{\op},\|B\|_{\op}\}}\ \|A-B\|_{\op}.
\]
Moreover,
\[
\|e^{A}-e^{B}\|_{\mathrm F}\le e^{\max\{\|A\|_{\op},\|B\|_{\op}\}}\ \|A-B\|_{\mathrm F}.
\]
\end{lemma}
\begin{proof}
The first statement follows from the series definition of the matrix exponential and submultiplicativity of the operator norm.

For the remaining statements, we include a proof (which we learned from \cite{vomEnde2024MSEMatrixExpDiff}) for completeness.
We first prove the representation
\begin{equation}\label{eq:exp_diff_integral}
e^{A}-e^{B}\ =\ \int_{0}^{1} e^{(1-u)A}\,(A-B)\,e^{uB}\,du.
\end{equation}
Define \(f(u):=e^{(1-u)A}e^{uB}\) for \(u\in[0,1]\). Differentiating gives
\[
f'(u)=-Ae^{(1-u)A}e^{uB}+e^{(1-u)A}Be^{uB}
= e^{(1-u)A}(B-A)e^{uB}.
\]
Since \(f(0)=e^{A}\) and \(f(1)=e^{B}\), we have
\(e^{A}-e^{B}=-\int_{0}^{1} f'(u)\,du\),
which yields \eqref{eq:exp_diff_integral}.

For the operator-norm bound, apply submultiplicativity, the first statement, and \eqref{eq:exp_diff_integral}:
\[
\|e^{A}-e^{B}\|_{\op}
\le \int_0^1 \|e^{(1-u)A}\|_{\op}\,\|A-B\|_{\op}\,\|e^{uB}\|_{\op}\,du
\le e^{\max\{\|A\|_{\op},\|B\|_{\op}\}}\,\|A-B\|_{\op}.
\]

For the Frobenius-norm bound, use the  inequality \(\|XYZ\|_{\mathrm F}\le \|X\|_{\op}\|Y\|_{\mathrm F}\|Z\|_{\op}\), together with \eqref{eq:exp_diff_integral} and the first statement of the lemma:
\[
\|e^{A}-e^{B}\|_{\mathrm F}
\le \int_0^1 \|e^{(1-u)A}\|_{\op}\,\|A-B\|_{\mathrm F}\,\|e^{uB}\|_{\op}\,du
\le e^{\max\{\|A\|_{\op},\|B\|_{\op}\}}\,\|A-B\|_{\mathrm F}.
\]
\end{proof}

The next lemma establishes a bound on the Lipschitz constant of the product of two (matrix-valued) functions.
\begin{lemma}[Lipschitz product rule in operator norm]\label[lemma]{lem:Lip_product_rule}
Let \((\Omega,d)\) be a metric space and let \(G,H:\Omega\to\C^{m\times m}\) be matrix-valued maps.
Assume \(\|G\|_{\infty},\|H\|_{\infty},\Lip(G)^{\op},\Lip(H)^{\op}<\infty\).
Then
\[
\Lip(GH)^{\op}\ \le\ \|G\|_{\infty}\,\Lip(H)^{\op}\ +\ \Lip(G)^{\op}\,\|H\|_{\infty}.
\]

\end{lemma}

\begin{proof}
For \(x\neq y\),
\[
G(x)H(x)-G(y)H(y)
=
G(x)\big(H(x)-H(y)\big)+\big(G(x)-G(y)\big)H(y).
\]
Taking operator norms and using submultiplicativity,
\[
\|G(x)H(x)-G(y)H(y)\|_{\op}
\le
\|G(x)\|_{\op}\,\|H(x)-H(y)\|_{\op}
+
\|G(x)-G(y)\|_{\op}\,\|H(y)\|_{\op}.
\]
Divide by \(d(x,y)\) and take the supremum over \(x\neq y\) to obtain
\[
\Lip(GH)^{\op}
\le
\|G\|_{\infty}\,\Lip(H)^{\op}+\Lip(G)^{\op}\,\|H\|_{\infty}.
\]
\end{proof}

%% file: kappa_squared.tex
\section{Concentration inequalities by Ollivier's method}
\label[appendix]{app:ollivier}
In this section we generalise the concentration inequalities of \citet{OllivierJFA} and \citet{JoulinOllivier} to the inhomogeneous and matrix setting.
We assume the same setup as in \cref{sec:curv}. In particular, we will be assuming Assumptions \ref{ass:curv}, \ref{ass:diam}, and \ref{ass:lip_obs}.

The following lemma follows by combining Lemma 7.6 and Lemma 4.3 of  \cite{userfriendly}.
\begin{lemma}
\label[lemma]{lem:tropp_step}
Let $P$ be a Markov operator on a Polish metric space $(\Omega,d)$. Let $G\colon \Omega \to \Hm$ and $x \in \Omega$ such that $PG(x) = 0$. Let $A \in \Hm$ such that $G(y)^2 \preceq A^2$ for any $y \in \supp P(x,\cdot)$.
Then, for any deterministic $H \in \Hm$,
\[
P \tr\exp\left( H + G \right)(x) \le  \tr\exp\left( H +  2 A^2 \right).
\]
\end{lemma}

We will bound the trace of the matrix moment generating function of $(1/T) \cdot \sum_{k=1}^T F_k(X_k)$ by an inductive argument.
The next lemma is the key building block of such argument.

\begin{lemma}
\label[lemma]{lem:induction}
Let $P$ be a Markov operator on a Polish metric space $(\Omega,d)$. Let $G\colon \Omega \to \Hm$ be $L$-Lipschitz
and $H \in  \Hm$. Let $\sigma_\infty(x) \coloneqq \diam \supp P(x,\cdot)$.
Then,
\[
P \tr\exp\left( H + G \right)(x) \le  \tr\exp\left(H + PG(x) + 2 L^2 \sigma_{\infty}^2(x) \Id \right).
\]
\end{lemma}
\begin{proof}
First of all we notice that, for any $y \in \supp P(x,\cdot)$, $\|G(y) - PG(x)\|_{\op} \le L \sigma_\infty(x)$.
Then, by \cref{lem:tropp_step},
\begin{align*}
P \tr\exp\left( H + G \right)(x) &= P \tr\exp\left( H + PG(x) + G - PG(x)  \right)(x) \\
&\le \tr\exp\left( H + PG(x) + 2L^2 \sigma_\infty^2(x) \Id \right).
\end{align*}
\end{proof}

The first main theorem of the section shows that positive curvature implies $F_T(X_T)$ concentrates around its expectation. This generalises Theorem 22 in \citep{OllivierJFA}.

\begin{theorem}
\label{thm:conc}
Let $F_T:\Omega\to\Hm$ be $L$-Lipschitz.  Let $\sigma_{\infty} \coloneqq \sup_{t \ge 1, x \in \Omega} \diam \supp P_t(x,\cdot)$.
Then, for any $T \ge 1$,
\[
\Pbb_{X_0}\!\left(\left\| F_T(X_T) - \E_{X_0}\!\left[F_T(X_T)\right]\right\|_{\op} \ge \varepsilon\right)
\le
2 m \, \exp\left(-\frac{\varepsilon^2 \kappa}{8L^2 \sigma_{\infty}^2 }\right)
\]
\end{theorem}
\begin{proof}
Let $s > 0$. We bound the matrix moment generating function
\begin{align*}
\E_{X_0}\!\left[\tr\exp(s F_T(X_T))\right]
&= \int_{\Omega^T} \tr \exp(s F_T(X_T)) \,
   P_T(X_{T-1},dX_T) \\
&\qquad\qquad\cdot P_{T-1}(X_{T-2},dX_{T-1}) \cdots P_1(X_0,dX_1) \\
&= \int_{\Omega^{T-1}} \left(P_{T} \tr \exp(s F_T)\right)(X_{T-1}) \, \\
&\qquad\qquad\cdot P_{T-1}(X_{T-2},dX_{T-1}) \cdots P_1(X_0,dX_1)
\end{align*}

Applying \cref{lem:induction} (with $P=P_T$, $H=0$, and $G=s F_T$) yields
\begin{align*}
&\E_{X_0}\!\left[\tr \exp(s F_T(X_T))\right] \\
&\,\le \int_{\Omega^{T-1}}\tr \exp\left(s (P_T F_T)(X_{T-1}) + 2s^2L^2\sigma_{\infty}^2\Id\right) \,  P_{T-1}(X_{T-2},dX_{T-1}) \cdots P_1(X_0,dX_1)\\
&\,= \int_{\Omega^{T-2}}\Bigl(P_{T-1}\tr \exp\left(s (P_T F_T) + 2s^2L^2\sigma_{\infty}^2\Id\right)\Bigr)(X_{T-2}) \,  P_{T-2}(X_{T-3},dX_{T-2}) \cdots P_1(X_0,dX_1).
\end{align*}

Applying \cref{lem:induction} a second time (with $P=P_{T-1}$ and $G=s\, P_T F_T$) and using the contraction
$\Lip(P_T F_T)^{\op} \le (1-\kappa_T)\,L$, we obtain
\begin{align*}
\E_{X_0}\!\left[\tr \exp(s F_T(X_T))\right]
&\le \int_{\Omega^{T-2}} \tr \exp\!
   \begin{aligned}[t]
   &\biggl(s (P_{T-1}P_T F_T)(X_{T-2}) + 2s^2L^2\sigma_{\infty}^2\Id \\[-0.2ex]
   &{}+ 2(1-\kappa_T)^2s^2L^2\sigma_{\infty}^2\Id\biggr)
   \end{aligned}
   \\
&\hspace{6.2em}\cdot P_{T-2}(X_{T-3},dX_{T-2}) \cdots P_1(X_0,dX_1).
\end{align*}

Iterating this argument and using
$\Lip(P_{T-k:T}F_T)^{\op} \le L\prod_{t=T-k}^T (1-\kappa_t)$, we obtain
\begin{align*}
\E_{X_0}\!\left[\tr\exp(s F_T(X_T))\right]
&\le \tr \exp\left(s (P_{1:T}F_T)(X_{0}) + 2s^2L^2 \sigma_{\infty}^2 \sum_{k=0}^{T-1} \prod_{t=T-k+1}^T (1-\kappa_t)^2\Id\right).
\end{align*}

By \cref{ass:curv},
\[
\E_{X_0}\!\left[\tr\exp(s F_T(X_T))\right]
\le \tr \exp\left(s \E_{X_0}\!\left[F_T(X_{T})\right] + 2s^2L^2 \sigma_{\infty}^2\kappa^{-1} \Id\right).
\]

Applying the previous moment bound to the centred function
\(\widetilde F_T := F_T-\E_{X_0}[F_T(X_T)]\) and then using Markov's inequality,
\[
\Pbb_{X_0}\!\left(\lmax\left(F_T(X_T) - \E_{X_0}\!\left[F_T(X_{T})\right] \right) \ge \varepsilon\right)
\le m \exp\left(-s \varepsilon + 2 s^2 L^2 \sigma_{\infty}^2 \kappa^{-1}\right).
\]

Choosing $s = \frac{\varepsilon \kappa}{4 L^2 \sigma_{\infty}^2}$ yields
\[
\Pbb_{X_0}\!\left(\lmax\left(F_T(X_T) - \E_{X_0}\!\left[F_T(X_{T})\right] \right) \ge \varepsilon\right)
\le m \exp\left(-\frac{\varepsilon^2 \kappa}{8L^2 \sigma_{\infty}^2 }\right).
\]

The statement follows by symmetry.

\end{proof}

We now generalise the concentration inequality for empirical means of \citet{JoulinOllivier}. For technical reasons, we will use a slightly different definition of effective curvature.

\begin{assumption}[Positive curvature, alternative definition]\label[assumption]{ass:curv2}
For all \(t\ge 1\), we have \(\kappa_t \in [0,1]\). Furthermore, there exists \(\tilde\kappa \in (0,1]\) such that, for all \(1\le s\le t\),
\begin{equation*}
1+\sum_{k=s}^{t}\prod_{\ell=s}^{k}(1-\kappa_\ell) \le\ \frac{1}{\tilde\kappa}.
\end{equation*}
\end{assumption}

\begin{theorem}
\label{thm:matrixMCMC}
Let $F_t:\Omega\to\Hm$ be $L$-Lipschitz for all $t \ge 1$.
\par\noindent
Let \(\sigma_{\infty} \coloneqq \sup_{t \ge 1, x \in \Omega} \diam \supp P_t(x,\cdot)\).
Then, for any $\varepsilon > 0$,
\[
\Pbb_{X_0}\!\left(\left\|\frac{1}{T}  \sum_{k=1}^T F_k(X_k) - \E_{X_0}\!\left[\frac{1}{T}  \sum_{k=1}^T F_k(X_k)\right]\right\|_{\op} >  \varepsilon\right)
\le 2m \exp\left(-\frac{\tilde\kappa^2 T \varepsilon^2 }{8L^2\sigma_\infty^2}\right).
\]
\end{theorem}
\begin{remark}
Notice that \cref{thm:matrixMCMC} does not require an explicit bound on the diameter. It does so, however, at a cost of a quadratic dependency on $\kappa^{-1}$.
\end{remark}
\begin{proof}
We bound the matrix moment generating function of $(1/T) \cdot \sum_{k=1}^T F_k(X_k)$ by an inductive argument.
In particular, we prove that, for any $1 \le t \le T$,
\begin{align}
&\E_{X_0}\!\left[\tr\exp\left(\frac{s}{T}  \sum_{k=1}^T F_k(X_k)\right)\right] \nonumber \\
&\qquad\le  \int_{\Omega^{T-t}} \tr\exp\left( \frac{s}{T}  \sum_{k=1}^{T-t} F_k(X_k) + \frac{s}{T} \sum_{j=1}^t P_{T-t + 1} \cdots P_{T-t + j} F_{T-t+j}(X_{T-t})
 + \frac{2 t s^2 L^2  \sigma_\infty^2}{\tilde\kappa^2T^2}  \Id\right) \nonumber \\
&\qquad \qquad \qquad  P_{T-t}(X_{T-t-1},dX_{T-t}) \cdots P_1(X_0,dX_1). \label{eq:induction}
\end{align}

We start with the base case $t=1$. We apply \cref{lem:induction} choosing $P = P_{T}$, $H  = (s/T) \sum_{k=1}^{T-1} F_k(X_k)$,  and $g = (s/T) F_T$. We obtain
\begin{align*}
&\E_{X_0}\!\left[\tr\exp\left(\frac{s}{T}  \sum_{k=1}^T F_k(X_k)\right)\right] \\
&\, =  \int_{\Omega^T} \tr\exp\left(\frac{s}{T}  \sum_{k=1}^T F_k(X_k) \right) \,
   P_T(X_{T-1},dX_T) P_{T-1}(X_{T-2},dX_{T-1}) \cdots P_1(X_0,dX_1) \\
&\,  =  \int_{\Omega^{T-1}} P_{T} \tr\exp\left(\frac{s}{T}  \sum_{k=1}^T F_k(X_k) \right) (X_{T-1}) P_{T-1}(X_{T-2},dX_{T-1}) \cdots P_1(X_0,dX_1) \\
&\,  \le  \int_{\Omega^{T-1}}  \tr\exp\!
   \begin{aligned}[t]
   &\biggl(\frac{s}{T}  \sum_{k=1}^{T-1} F_k(X_k)
   + \frac{s P_{T}F_{T}(X_{T-1})}{T}  + 2L^2 \sigma_\infty^2 s^2 T^{-2} \Id\biggr)
   \end{aligned}
   \, \\
&\qquad\qquad\cdot P_{T-1}(X_{T-2},dX_{T-1}) \cdots P_1(X_0,dX_1).
\end{align*}

Suppose now \cref{eq:induction} holds for $1 \le t \le T-1$.
We apply \cref{lem:induction} with $P = P_{T-t}$ and $g  = \frac{s}{T}  F_{T-t} + \frac{s}{T} \sum_{j=1}^{t} P_{T-t + 1 : T-t + j} F_{T-t+j}$.
Notice that, by \cref{ass:curv2},
\[
\Lip(g)
\le
\frac{sL}{T}
\left(
1+\sum_{j=1}^{t}
\prod_{i=T-t+1}^{T-t+j}(1-\kappa_i)
\right)
\le \frac{sL}{\tilde\kappa T}.
\]
Therefore,
\begin{align*}
&\E_{X_0}\!\left[\tr\exp\left(\frac{s}{T}  \sum_{k=1}^T F_k(X_k)\right)\right]  \\
&\, \le  \int_{\Omega^{T-t}} \tr\exp\!
   \begin{aligned}[t]
   &\biggl(\frac{s}{T}  \sum_{k=1}^{T-t - 1} F_k(X_k)
   + \frac{s}{T}  F_{T-t}(X_{T-t}) \\[-0.2ex]
   &{}+ \frac{s}{T} \sum_{j=1}^t P_{T-t+1} \cdots P_{T-t+j} F_{T-t+j}(X_{T-t})
   + \frac{2 t L^2 s^2  \sigma_\infty^2}{\tilde\kappa^2T^2}  \Id\biggr)
   \end{aligned}
   \\
 &\, \qquad \qquad  P_{T-t}(X_{T-t-1},dX_{T-t}) \cdots P_1(X_0,dX_1) \\
 &\,\le \int_{\Omega^{T-t-1}} \tr\exp\!
   \begin{aligned}[t]
   &\biggl(\frac{s}{T}  \sum_{k=1}^{T-t - 1} F_k(X_k) 
   +  \frac{s}{T} \sum_{j=1}^{t+1}
   P_{T-t} \cdots P_{T-t-1+j} F_{T-t-1+j}(X_{T-t-1}) \\[-0.2ex]
   &{}+ \frac{2 L^2 s^2  \sigma_\infty^2}{\tilde\kappa^2T^2} \Id
   + \frac{2 t L^2 s^2  \sigma_\infty^2}{\tilde\kappa^2T^2}  \Id\biggr)
   \end{aligned}
   \\
 &\, \qquad \qquad  P_{T-t-1}(X_{T-t-2},dX_{T-t-1}) \cdots P_1(X_0,dX_1),
\end{align*}
which proves \cref{eq:induction}.

After $t=T$ steps, \cref{eq:induction} yields
\begin{align*}
\E_{X_0}\!\left[\tr\exp\left(\frac{s}{T}  \sum_{k=1}^T F_k(X_k)\right)\right]
&\le \tr\exp\left(\frac{s}{T} \sum_{k=1}^T P_{1} \cdots P_{k} F_k(X_0) + \frac{2 s^2  L^2 \sigma_\infty^2}{\tilde\kappa^2T}  \Id \right) \\
&= \tr\exp\left(\E_{X_0}\!\left[\frac{s}{T} \sum_{k=1}^T F_k(X_k)\right] + \frac{2 s^2  L^2 \sigma_\infty^2}{\tilde\kappa^2T}  \Id  \right).
\end{align*}

Apply the preceding moment bound to the centred observables
\(\widetilde F_k := F_k-\E_{X_0}[F_k(X_k)]\). Set $s = \frac{\varepsilon \tilde\kappa^2 T}{4 L^2\sigma_\infty^2}$ and use Markov's inequality to obtain
\begin{align*}
&\Pbb_{X_0}\!\left(\lmax\left(\frac{1}{T} \cdot \sum_{k=1}^T F_k(X_k) - \E_{X_0}\!\left[\frac{1}{T}  \sum_{k=1}^T F_k(X_k)\right]\right) >  \varepsilon\right) \\
&\qquad\qquad \le \exp(-s \varepsilon) \E_{X_0}\!\left[\tr\exp\left(\frac{s}{T}  \sum_{k=1}^T F_k(X_k) - \E_{X_0}{\frac{s}{T} \sum_{k=1}^T F_k(X_k)}\right)\right] \\
&\qquad\qquad\le \exp(-s \varepsilon) \tr\exp\left(\frac{2 s^2 L^2 \sigma_\infty^2}{\tilde\kappa^2T}  \Id  \right) \\
&\qquad\qquad\le  m \exp\left(\frac{2 s^2 L^2  \sigma_\infty^2}{\tilde\kappa^2 T} -s \varepsilon\right) \\
&\qquad\qquad\le m \exp\left(-\frac{\tilde\kappa^2 T \varepsilon^2 }{8L^2\sigma_\infty^2}\right).
\end{align*}
The statement of the theorem follows by considering $(-F_t)_t$ as well.
\end{proof}